\documentclass[12pt]{amsart}
\usepackage{amsmath, amsfonts,graphicx}
\usepackage{color}

\numberwithin{equation}{section}
\newtheorem{theorem}{Theorem}[section]
\newtheorem{corollary}[theorem]{Corollary}
\newtheorem{lemma}[theorem]{Lemma}
\newtheorem{proposition}[theorem]{Proposition}

\newtheorem{definition}{Definition}[section]
\newtheorem{remark}{Remark}[section]
\newtheorem{claim}[theorem]{Claim}

\newcommand{\R}{\mathbb{R}}

\newcommand\be{\begin{equation}}
\newcommand\ee{\end{equation}}
\newcommand\bea{\begin{eqnarray}}
\newcommand\eea{\end{eqnarray}}
\newcommand\beaa{\begin{eqnarray*}}
\newcommand\eeaa{\end{eqnarray*}}

\begin{document}

\title[spreading speeds]{Asymptotic spreading speeds for a predator-prey system with two predators and one prey}

\author[A. Ducrot]{Arnaud Ducrot}
\address{UNIHAVRE, LMAH, FR-CNRS-3335, ISCN, Normandie University, 76600 Le Havre, France}
\email{arnaud.ducrot@univ-lehavre.fr}

\author[T. Giletti]{Thomas Giletti}
\address{Institut Elie Cartan Lorraine, UMR 7502, University of Lorraine, 54506 Vandoeuvre-l\`es-Nancy, France}
\email{thomas.giletti@univ-lorraine.fr}

\author[J.-S. Guo]{Jong-Shenq Guo}
\address{Department of Mathematics, Tamkang University, Tamsui, New Taipei City 251301, Taiwan}
\email{jsguo\,@\,mail.tku.edu.tw}

\author[M. Shimojo]{Masahiko Shimojo}
\address{Department of Applied Mathematics,  Okayama University of Science, Okayama 700-0005, Japan}
\email{shimojo@xmath.ous.ac.jp}

\thanks{Date: \today. Corresponding author: J.-S. Guo}



\begin{abstract}
This paper investigates the large time behaviour of a three species reaction-diffusion system, modelling the spatial invasion of two predators feeding on a single prey species.
In addition to the competition for food, the two predators exhibit competitive interactions and under some parameter conditions ($\mu>0$), they can also be considered as two mutants.
When mutations occur in the predator populations, the spatial spread of invasion takes place at a definite speed, identical for both mutants.
When the two predators are not coupled through mutation, the spreading behaviour exhibits a more complex propagating pattern, including multiple layers with different speeds.
In addition, some parameter conditions reveal situations where a nonlocal pulling phenomenon occurs and in particular where the spreading speed is not linearly determined.
\end{abstract}

\maketitle

\section{Introduction}
\setcounter{equation}{0}

In this work we study the large time behaviour for the following three-species predator-prey system
\bea
&&u_t=d_1 u_{xx}+uF (u,v,w) + \mu (v-u) \quad \mbox{in} \quad  (0,+\infty) \times \mathbb{R},\label{u-eq}\\
&&v_t=d_2 v_{xx}+vG (u,v,w) + \mu (u-v) \quad \mbox{in} \quad  (0,+\infty) \times \mathbb{R},\label{v-eq}\\
&&w_t=d_3 w_{xx}+wH (u,v,w) \quad \mbox{in} \quad  (0,+\infty) \times \mathbb{R},\label{w-eq}
\eea
wherein the functions $F$, $G$ and $H$ are respectively given by
\beaa
&&F (u,v,w) :=r_1(-1-u-kv+aw),\\
&&G (u,v,w) :=r_2(-1-hu-v+aw),\\
&&H (u,v,w) :=r_3(1-bu-bv-w).
\eeaa
Here $u=u(t,x)$ and $v=v(t,x)$ denote the densities of two predators and $w=w(t,x)$ corresponds to the density of the single prey.
The parameters $d_i$, $r_i$ in system \eqref{u-eq}-\eqref{w-eq} are all positive and respectively stand for the motility and per capita growth rate of each species.
The parameters $h>0$ and $k>0$ represent the competition between the two predators while $a>0$ and $b>0$ describe the predator-prey interactions.
The parameter $\mu$ is assumed to be nonnegative.

When $\mu>0$ it stands as a mutation rate between species $u$ and $v$, so that the two predators are mutant species.
Therefore, one should distinguish the case $\mu = 0$, where $u$ and $v$ can be understood as completely distinct species which are competing for the same prey,
from the case $\mu >0$ where $u$ and $v$ are the densities of two mutants of the same predator species, both feeding on the same prey. Both situations also lead to very different large time behaviours of the solutions.
We shall refer to the former as the `two competitors' case, and to the latter as the `two mutants' case.

To perform our study of system \eqref{u-eq}-\eqref{w-eq}, we impose the following parameter conditions
\be\label{c1}
a>1,\quad 0<h,k<1,\quad 0<b<\frac{1}{2(a-1)} , \quad  0 \leq \mu \leq  \frac{a-1}{2} \min \{ r_1, r_2 \}.
\ee
Condition~\eqref{c1} shall be assumed throughout this paper, and as we shall see below it typically allows the possibility of co-existence of all three species.

The main goal of this paper is to study the asymptotic spreading speed(s) of the two invading predators into a prey population uniformly close to its carrying capacity.
For this purpose, we typically investigate the Cauchy problem for \eqref{u-eq}-\eqref{w-eq} with the initial condition
\be\label{ic}
u(0,x)=u_0(x),\quad v(0,x)=v_0(x),\quad w(0,x)=w_0(x),\quad x\in\R,
\ee
where initial data $u_0, v_0, w_0$ are uniformly continuous functions defined on $\R$, both $u_0$ and $v_0$ have nontrivial compact supports, and
\be\label{c2}
0\le u_0,\; v_0\le a-1,\quad \beta\le w_0\le 1\;\mbox{ on $\R$.}
\ee
Hereafter, we set $\beta:=1-2 b(a-1) >0$ in which the positivity follows from the third condition in~\eqref{c1}.
Note that by the classical theory of reaction-diffusion systems there is a unique global (in time) solution $(u,v,w)$
of the Cauchy problem for system \eqref{u-eq}-\eqref{w-eq} with initial condition \eqref{ic} for any uniformly continuous initial data $(u_0,v_0,w_0)$ satisfying~\eqref{c2}.

The description of spatial propagation for diffusive predator-prey systems has a long history. In particular, traveling wave solutions have been exhibited in a wide range of predator-prey systems.
For the systems with exactly one prey and one predator, we refer in particular to the pioneering works~\cite{Dunbar,Gardner}. We also refer to \cite{Huang-Lu-Ruan,Hsu-etal,Mischaikow}
and the references cited therein for more results about traveling waves in such a context.
Recently, more attention is paid to the existence of traveling waves of some three species predator-prey systems.
For this, we refer the reader to, e.g., \cite{BP18,Huang-Lin} for a diffusive system involving two preys and one predator,
and \cite{GNOW20,LWZY15} for a system with two predators and one prey.

While some stability results were obtained in~\cite{GJ91}, it is only recently that a more exhaustive study of the spreading properties of solutions of the Cauchy problem has been performed
for systems with one prey and one predator~\cite{D,DGM,DGLP,Pan}.
This naturally raises the question of propagation phenomena in more realistic systems involving a larger number of interacting species.
Let us mention in particular that the spreading speed of the predator for a three-species predator-prey system with a single predator and two preys was studied by Wu~\cite{Wu}.

As far as mutant systems are considered, we refer to Griette and Raoul~\cite{GR16} for a study of traveling wave solutions.
We also refer to Girardin~\cite{Girardin18} for a result on the asymptotic spreading speed of solutions of the Cauchy problem and to Morris et al~\cite{Morris} for discussion about the linear determinacy of the spreading speed.
However, those earlier works were concerned with systems coupling mutations and competitive interactions between the mutants.
As far as we know, our work is the first to investigate the spreading properties in a situation involving both mutations and prey-predator interactions.

One important difficulty to be overcome in our analysis of the spreading properties for system \eqref{u-eq}-\eqref{w-eq} with \eqref{ic} is the lack of comparison principle.
In addition, as described below, the `two-competitors' system admits different semi-trivial equilibria that generate complex propagating behaviour,
including the development of multiple propagating layers with different speeds.
These difficulties are overcome by using refined dynamical system arguments and by carefully investigating suitable omega-limit solutions in various and adapted moving frames.

The rest of this paper is organized as follows. First, Section~2 is devoted to the description of the main results obtained in this paper.
This comes after introducing some notation to be used throughout this paper.
Then we start in Section~\ref{sec:prelim} with some preliminaries.
These preliminaries include a priori estimates on solutions and the computation of the principal eigenvalues of some linearized systems which shall arise in the proof of our main results.
Moreover, we establish some persistence result, Lemma~\ref{LE-hyp}, which shall be a key ingredient in our arguments, in the spirit of~\cite{DGM}.
Then, in Section~\ref{sec:lyapunov}, we show by a Lyapunov like argument that, in the `two competitors case', bounded entire in time solutions which are bounded from below by some positive constants
must be identical to a steady state.

The later sections are directly concerned with the proofs of our main results {described in Section 2.}
Section~\ref{sec:mutants} mostly deals with 
the `two mutants' case and some general upper bounds on the spreading speeds, both for the `two mutants' and `two competitors' cases.
Next, in Section~\ref{sec:competitors1} we establish some of the lower bounds on the speeds in the `two competitors' case. 
Then in Section~\ref{sec:competitors2} we {deal with} 
the case when both predator species have the same speed and diffusivity.
Finally, 
in Section~\ref{competitors_faster}, {we first derive the definite spreading speed of the faster predator in the `two competitors' case.
Then we exhibit a `nonlocal pulling' phenomenon for the slower predator. This leaves a question on the definite spreading speed of the slower predator to be open.}


\section{Main results}

In this section, we state the main results that shall be proved and discussed in this paper.
We start by introducing some notation and important quantities before going to the description of our main spreading results.

\subsection{Some notation}

\subsubsection{Constant equilibria}

Let us first look at the underlying ODE system. It is clear that $(0,0,0)$ and $(0,0,1)$ ({\it the predator-free state}) are two trivial equilibria for system \eqref{u-eq}-\eqref{w-eq}; both are linearly unstable.
For other steady states, we consider separately the `two competitors' and the `two mutants' cases.
In both cases, we shall find that the three components ultimately persist in the large time behaviour of solutions of~\eqref{u-eq}-\eqref{w-eq}.
Furthermore, in the first case, the solution converges to the unique co-existence steady state which we compute below.

\noindent{\bf Two competitors case $\mu = 0$:} It is easy to check that, when $a>1$, $(0,\tilde{p},\tilde{q})$ and $(\tilde{p},0,\tilde{q})$ are two semi-trivial steady states of system \eqref{u-eq}-\eqref{w-eq}, where
\be\label{def-pq}
\tilde p=\frac{a-1}{1+ab},\quad \tilde{q}=\frac{b+1}{1+ab}.
\ee
Moreover, under condition \eqref{c1} there is a unique positive co-existence state $(u^*,v^*,w^*)$, where
\begin{equation}\label{def-equi}
\begin{split}
&u^*:=\frac{1-k}{1-hk}(aw^*-1),\quad v^*:=\frac{1-h}{1-hk}(aw^*-1),\\
&w^*:=\frac{(1-hk)+b(2-h-k)}{(1-hk)+ab(2-h-k)}.
\end{split}
\end{equation}
Note that one has
\beaa
F(0,\tilde p,\tilde q)=\frac{r_1}{1+ab}(a-1)(1-k)>0,\quad G(\tilde p,0,\tilde q)=\frac{r_2}{1+ab}(a-1)(1-h)>0.
\eeaa
This means that the semi-trivial steady states are unstable with respect to the underlying kinetic ODE system.

\noindent{\bf Two mutants case $\mu >0$:} In this case, because of the coupling between the two mutants it is clear that there cannot exist a semi-trivial steady state of system \eqref{u-eq}-\eqref{w-eq}.
However, as in the case $\mu = 0$, there typically exists a unique positive co-existence steady state $(u^*_\mu , v^*_\mu, w^*_\mu)$. This can be checked proceeding as in \cite{GR16} provided that $\mu$ is small enough.
Note that, as discussed in \cite{Coville-Fabre}, the mutation rate $\mu$ is typically small for relevant biological situations. Since a more detailed analysis of the equilibria is not necessary to our purpose, we omit the details.


\subsubsection{Linear speeds}

We shall take an interest in the invasion of predators on the predator-free state. Consider the linearization of \eqref{u-eq} and \eqref{v-eq} around the predator-free state $(0,0,1)$, and write it in matrix form as
$$\left( \begin{array}{c} u_t \\ v_t\end{array} \right) = \left( \begin{array}{c} d_1 u_{xx} \\ d_2 v_{xx}\end{array} \right) + M [\mu] \times \left( \begin{array}{c} u \\ v\end{array} \right),$$
where
$$ M [\mu]  := \left[ \begin{array}{cc}
r_1 (a-1) - \mu & \mu \\
\mu & r_2 (a-1) - \mu
\end{array}
\right].$$
In the case when $\mu >0$, the components $u$ and $v$ are mutants of the same species and therefore they should spread simultaneously. By analogy with the scalar equation, one may want to look for an ansatz of the type
$$ (u,v)= (p, q) e^{-\gamma (x-ct)}, \quad p, q , \gamma  >0,\; c\in\R.$$
Putting this in the above equation, one finds that $(p, q)$ and $\gamma$ must solve
$$ M [\mu , \gamma] \times \left( \begin{array}{c} p \\ q \end{array} \right) = c \gamma \left( \begin{array}{c} p \\ q\end{array} \right),$$
where
\begin{equation}\label{matrix_mutant}
 M [\mu ,\gamma ]  := \left[ \begin{array}{cc}
d_1 \gamma^2 + r_1 (a-1) - \mu & \mu \\
\mu & d_2 \gamma^2 + r_2 (a-1) - \mu
\end{array}
\right].
\end{equation}
This leads us to introduce
\begin{equation}\label{cmu*}
 c_\mu^* := \min_{\gamma >0} \frac{ \Lambda [\mu,\gamma] }{\gamma} >0 ,
 \end{equation}
such that an ansatz of the above type solves the linearized system around the predator-free state if and only if $c \geq c_\mu^*$. Here $\Lambda [\mu ,\gamma]$ denotes the unique principal eigenvalue of $M[\mu, \gamma]$,
which exists by applying Perron-Frobenius theorem. We point out that the existence and positivity of $c_\mu^*$ are ensured by the fact that $\Lambda [\mu,0] >0$ (recall that $2 \mu \leq  (a-1) \min \{ r_1, r_2 \}$)
and $\gamma \mapsto \Lambda [\mu,\gamma]$ is convex.

In the case when $\mu = 0$, the functions~$u$ and~$v$ are population densities of distinct species which may spread with different speeds. This can be seen in the fact that the matrix $M[0]$ is no longer irreducible;
one can then find semi-trivial ansatzes
$$ (u,v) = (p, 0) e^{-\gamma (x-ct)}, \quad  (u, v) = (0 ,q) e^{-\gamma (x-ct)}, \quad p, q >0.$$
These exist if, respectively, $c \geq c^*_u$ and $c \geq c^*_v$. Those speeds are explicitly defined as
\begin{equation}\label{c*}
c_u^*:=2\sqrt{d_1 r_1(a-1)},\quad c_v^*:=2\sqrt{d_2 r_2(a-1)}.
\end{equation}
One may check that $c^*_\mu$ converges to $\max \{ c_u^* , c_v^* \}$ as $\mu \to 0$. Without loss of generality, we may assume that $d_1r_1\ge d_2r_2$. Hence throughout this paper we always assume that
$$c_u^*\ge c_v^* .$$
In the case when $\mu = 0$, it is then expected that the component~$u$ shall spread with speed $c_u^*$, and the component~$v$ at some slower speed. However, this means that $v$ may no longer invade in the predator-free state.
This leads us to introduce
\begin{equation}\label{c**}
\begin{split}
&c_u^{**}:=2\sqrt{\frac{d_1r_1}{1+ab}(a-1)(1-k)}=c_u^*\sqrt{\frac{1-k}{1+ab}},\\
&c_v^{**}:=2\sqrt{\frac{1-h}{1+ab}d_2r_2(a-1)}=c_v^*\sqrt{\frac{1-h}{1+ab}}.
\end{split}
\end{equation}
Note that $c_v^{**}<c_v^*$ and $c_v^{**}$ is the spreading speed of solutions of
\beaa
v_t = d_2 v_{xx} + r_2 v (-1 - h \tilde{p} + a \tilde{q}),
\eeaa
i.e., the linear invasion speed into the semi-trivial steady state (when $\mu = 0$) where $u$ and $w$ coexist.

\subsection{Our main results}

We are now in a position to describe the spreading speed of solutions.
In the sequel, we let $(u,v,w)$ be a solution of the Cauchy problem for system \eqref{u-eq}-\eqref{w-eq} with the initial condition \eqref{ic}, under the assumption \eqref{c2}.

Our first theorem deals with the `two mutants' case $ \mu >0$ where we can describe accurately the spreading speed of the solution:

\begin{theorem}\label{th:mutants}
Assume that~\eqref{c1} holds and that in addition $\mu >0$. Let $(u_0 , v_0)$ be nontrivial and compactly supported such that $ 0 \leq u_0 , v_0 \leq a-1 $, and $\beta \leq w_0 \leq 1$.
Recall also that $c^*_\mu$ is defined by~\eqref{cmu*}.
Then the solution $(u,v,w)$ enjoys the following spreading behaviour:
\begin{itemize}
\item [(i)] for any $c > c^*_\mu$,
$$\lim_{t \to +\infty} \sup_{|x| \geq ct} \{|u (t,x)| + |v (t,x)| +  | 1 - w(t,x)|\}= 0;$$
\item [(ii)] for any $c\in[0,c^*_\mu)$,
$$\liminf_{t \to +\infty} \inf_{|x| \leq ct}  \min \{ u (t,x) ,  v (t,x) , 1 - w(t,x) \} >0.
$$
\end{itemize}
\end{theorem}
The above result shows that both predator mutants spread with the same asymptotic speed. However, the question of the convergence to a steady state in the wake of the propagation remains open.

Let us now turn to the `two competitors' case $\mu = 0$. This situation turns out to be more complicated,
because unlike in the `two mutants' case the subsystem~\eqref{u-eq} and~\eqref{v-eq} with constant $w$ has no cooperative structure in the neighborhood of $(0,0)$.
Moreover, both competitors may spread with different spreading speeds, so that ultimately one may observe three zones: ahead of the propagation in the predator-free state,
after the propagation of the first predator but before arrival of the second, and finally a co-existence state where all three species persist.

The next theorem provides some estimates on the spreading speeds of both species.

\begin{theorem}\label{theo:spread_plus}
Assume that~\eqref{c1} holds and that in addition $\mu = 0$. Let both $u_0$ and~$v_0$ be nontrivial and compactly supported such that $0 \leq u_0, v_0 \leq a-1$, and $ \beta \leq w_0 \leq 1$.
Recall also that $c_u^*$ and $c_v^*$ are defined in \eqref{c*}.
Then the solution $(u,v,w)$ satisfies, for any $c > c_u^*$,
\beaa
\lim_{ t \to +\infty} \sup_{|x| \geq ct} u(t,x) =0,
\eeaa
and, for any $c>c_v^*$,
\beaa
\lim_{ t \to +\infty} \sup_{|x| \geq ct} v(t,x) =0.
\eeaa
Moreover, for any $c > \max \{ c_u^*, c_v^*\}$,
\beaa
\lim_{t \to +\infty} \sup_{|x| \geq ct } | 1 - w(t,x)| = 0.
\eeaa
\end{theorem}

In particular, this result states that $u$ and $v$ spread at most, respectively, with speeds~$c_u^*$ and~$c_v^*$. Now the question is whether $c_u^*$ and $c_v^*$ are indeed the spreading speed of $u$ and $v$.
To answer this point, we consider separately the cases when $c_v^* = c_u^*$ and $c_v^* < c_u^*$.
\begin{theorem}\label{THEO-inside1}
Assume that~\eqref{c1} holds and that in addition $\mu = 0$ and $c_v^* = c_u^*$. Let both $u_0$ and $v_0$ be nontrivial and compactly supported such that $0 \leq u_0, v_0 \leq a-1$, and $ \beta \leq w_0 \leq 1$.
Recall also that $c_u^{**}$ and $c_v^{**}$ are defined in~\eqref{c**}.
Then the solution $(u,v,w)$ enjoys the following spreading behaviour:
\begin{itemize}
\item [(i)] for each $c\in (0,c_v^*)$ one has
\begin{equation*}
\liminf_{t\to +\infty}\inf_{|x|\leq ct}\left(u+v\right)(t,x)>0;
\end{equation*}
\item [(ii)] the functions $u$ and $v$ separately satisfy
\begin{equation*}
\begin{split}
&\liminf_{t\to +\infty}\inf_{|x|\leq ct}u(t,x)>0,\;\;\forall c\in \left(0, c_u^{**}\right),\\
&\liminf_{t\to+\infty}\inf_{|x|\leq ct}v(t,x)>0,\;\;\forall c\in (0,c_v^{**});
\end{split}
\end{equation*}
\item [(iii)] finally, for each $0<c<\min (c_v^{**},c_u^{**})$,
\begin{equation*}
\lim_{t\to+\infty} \sup_{|x|\leq ct}\left\|(u,v,w)(t,x)-(u^*,v^*,w^*)\right\|=0,
\end{equation*}
wherein $(u^*,v^*,w^*)$ is the constant equilibrium defined in~\eqref{def-equi} and $\|\cdot\|$ denotes any norm in $\R^3$.
\end{itemize}
\end{theorem}

In the equi-diffusion case, the previous result can be strengthened as follows. In the sequel, we set $\kappa:=(1-k)/(1-h)$.

\begin{theorem}\label{inside}
Under the same assumptions as Theorem~\ref{theo:spread_plus}, suppose also that $w_0\equiv 1$, $d_1 = d_2$ and $r_1 = r_2$, so that in particular $c_u^* = c_v^*$.
Then, for any $c \in [0,c_u^*)$,
\[
\liminf_{t \to +\infty} \inf_{|x|\le ct} u(t,x) > 0,
\]
 if $u_0\ge\kappa v_0$;
while, for any $c\in [0,c_v^*)$, if $u_0\le\kappa v_0$ we have
\[
\liminf_{t \to +\infty} \inf_{|x|\le ct}  v(t,x) > 0.
\]
In each case, we also have
\[
\limsup_{t \to +\infty} \sup_{|x|\le ct} w(t,x)< 1.
\]
\end{theorem}

We now turn to the case when $c_v^* < c_u^*$, where the following result answers positively that $c_u^*$ is indeed the spreading speed of $u$.

\begin{theorem}\label{th:middle_zone}
Assume that~\eqref{c1} holds and that in addition $\mu = 0$ and $c_v^* < c_u^*$. Let both $u_0$ and $v_0$ be nontrivial and compactly supported such that $0 \leq u_0, v_0 \leq a-1$, and $ \beta \leq w_0 \leq 1$.
Then the solution $(u,v,w)$ enjoys the following spreading behaviour:
\begin{itemize}
\item [(i)] for each $c \in [0,c_u^*)$, one has
\begin{equation*}
\liminf_{t \to +\infty} \inf_{|x|\le ct} u(t,x) > 0;
\end{equation*}
\item [(ii)] for any $c_v^*<c_1<c_2<c_u^*$,
\begin{equation*}
\lim_{t\to+\infty} \sup_{c_1t\leq |x|\leq c_2 t}\left\|(u,v,w)(t,x)-\left(\tilde p, 0,\tilde q\right)\right\|=0,
\end{equation*}
wherein $\left(\tilde p, 0,\tilde q\right)$ is the constant equilibrium defined in \eqref{def-pq};
\item [(iii)] finally, for each $0<c<c_v^{**}$,
\begin{equation*}
\lim_{t\to+\infty} \sup_{|x|\leq ct}\left\|(u,v,w)(t,x)-(u^*,v^*,w^*)\right\|=0.
\end{equation*}
\end{itemize}
\end{theorem}

Together with Theorem~\ref{theo:spread_plus}, this describes accurately the spreading of the faster predator when $c_u^* > c_v^*$.
However, it is still only known that $v$ spreads at least at speed $c_v^{**}$, and at most at speed $c_v^*$. It remains an open question whether $v$ has a definite spreading speed and what this spreading speed is.
While one may expect that the spreading speed of $v$ is $c_v^{**}$, which indeed corresponds to the linear invasion speed into the intermediate semi-trivial steady state $(\tilde{p},0,\tilde{q})$,
we can actually construct a situation where the spreading speed is strictly faster than $c_v^{**}$. This is comparable to the nonlocally pulled phenomenon which occurs in competition systems~\cite{GL}.

\begin{theorem}[Nonlocal Pulling]\label{th-NP}
Under the same assumptions as in Theorem \ref{th:middle_zone}, assume furthermore that
\begin{equation*}
\begin{split}
&{1-2ab-h>0,}\\
&c_v^{**} - \sqrt{\left(c_v^{**}\right)^2 - 4 d_2 r_2 (a\beta - 1 - h (a-1))} >  \frac{\left(c_u^*\right)^2 - \left(c_v^*\right)^2}{2(c_u^* - c_v^{**}) }.
\end{split}
\end{equation*}
Then there exists $c>c_v^{**}$ (independent of the initial data) such that
$$
\liminf_{t\to +\infty} \inf_{|x|\leq ct}v(t,x)>0.
$$
\end{theorem}

\begin{remark}
The conditions proposed above hold in particular when $1-2ab -h >0$ and $0<d_1r_1-d_2 r_2$ is sufficiently small.
This ensures that there is a set of parameters which satisfy the assumptions of Theorem~\ref{th-NP}, thus the nonlocal pulling phenomenon does indeed occur.
\end{remark}

\section{Preliminaries}\label{sec:prelim}

\subsection{Some a priori estimates}

Let $(u,v,w)$ be a solution of the Cauchy problem for system \eqref{u-eq}-\eqref{w-eq} with the initial condition \eqref{ic}, under the assumption \eqref{c2}.
It is clear that the solution $(u,v,w)$ exists globally.
By the (strong) comparison principle (for scalar equations), it is easy to see that $u,v,w>0$ in $(0,\infty) \times \R$, and $w<1$ in $(0,\infty) \times \R$.
Then
\beaa
\begin{cases}
u_t \leq d_1 u_{xx} + r_1 u (a-1 -u ) + \mu (v-u),\\
v_t \leq d_2 u_{xx} + r_2 v (a-1  - v) + \mu (u-v).
\end{cases}
\eeaa
This system is of the cooperative type and satisfies a comparison principle. It is then straightforward that $u,v<a-1$ in $(0,\infty) \times \R$.
Moreover, by \eqref{c2} and \eqref{w-eq}, another comparison principle gives that $w\ge\beta$ in $(0,\infty) \times \R$.

We sum up the above in the following proposition.

\begin{proposition}\label{prop:prelim_estim}
Assume that $(u,v,w)$ solves \eqref{u-eq}-\eqref{w-eq} together with~\eqref{ic} with $\mu \geq 0$, where $0 \leq u_0 ,v_0 \leq a-1$, $\beta \leq w_0 \leq 1$.
Then the inequalities $0 \leq u, v \leq a-1$, $\beta \leq w \leq 1$ also hold for all $t >0$.
\end{proposition}

\subsection{Orbit closure sets and key persistence lemmas}

In this section and for convenience, we introduce
\begin{equation}\label{X0}
X_0 := \{ (u_0, v_0,w_0) \in  (UC^0 (\R ; \R))^3  \, | \ 0 \leq u_0, \ v_0 \leq a -1 , \ \beta \leq w_0 \leq 1 \},
\end{equation}
where $UC^0 (\R; \R)$ denotes the set of uniformly continuous and bounded functions from~$\R $ to $\R$. In other words, $X_0$ denotes the set of initial data satisfying~\eqref{c2}.

According to Proposition~\ref{prop:prelim_estim}, for any $(u_0,v_0,w_0) \in X_0$, the associated solution $(u,v,w)$ satisfies the same inequalities for all positive times, i.e.,
$$0 \leq u , v\leq a -1, \quad \beta \leq w \leq 1.$$
We also observe the following property. Let $(u,v,w)$ be a solution of \eqref{u-eq}-\eqref{w-eq} with initial data $(u_0,v_0,w_0 ) \in X_0$.
If $\mu >0$ and there exists $(t_0,x_0)\in\R^2$ with either $u (t_0,x_0)=0$ or $v (t_0,x_0)=0$, then, by the strong maximum principle,
$$
u(t,x)=0\text{ and } v(t,x)=0\text{ for all }(t,x)\in\R^2.
$$
On the other hand, if $\mu=0$ and there exists $(t_0,x_0)\in\R^2$ with $u (t_0,x_0)=0$  (resp. $v (t_0,x_0)=0$), then
$$
 u(t,x)=0\;\;(\text{resp. } v(t,x)=0)\text{ for all }(t,x)\in\R^2.
$$

In order to state our key lemma, it is also convenient to introduce some new function set $\omega_0 (c_2,c_1)$,
which roughly stands for the closure (with respect to the locally uniform topology) of the orbits of the solution in moving frames with speeds in the interval $[c_2,c_1]$. More precisely:
\begin{definition}\label{definition-omega0}
For any $(u_0,v_0,w_0) \in X_0$, and any $0 \leq c_2 < c_1$, we define the set
$$\omega_0 \left( c_2,c_1 \right) := closure \, \{ (u,v,w) (t , \cdot + ct) \, | \ t \geq 0, \ c \in [c_2,c_1] \} \subset X_0 ,$$
where $(u,v,w)$ denotes the solution of \eqref{u-eq}-\eqref{w-eq} with the initial data $(u_0,v_0,w_0)$.
Herein the closure is understood with respect to the locally uniform topology in $\R$.
\end{definition}
We point out that the fact that $\omega_0 (c_2,c_1)$ is a subset of $X_0$ is a straightforward consequence of the above discussion together with standard parabolic estimates.
In the sequel we have to keep in mind that this set depends on the initial data and, for notational convenience we omit to explicitly write down the dependence with respect to the initial data.


We now turn to some important lemma that shall be used several times throughout this manuscript.
\begin{lemma}\label{LE-weak-strong}
Assume that $(u_0,v_0,w_0) \in X_0$. Let $0 \leq c_2 < c_1 $ and $\zeta, \xi \geq 0$ be given such that $\zeta+\xi>0$. 
Assume that for  any $c\in [c_2,c_1 )$ there exists $\varepsilon (c) >0$ such that for any $(\tilde u_0,\tilde v_0,\tilde w_0) \in \omega_0(c_2,c_1)$ with  $\zeta \tilde u_0 + \xi \tilde v_0  \not \equiv 0$,
the corresponding solution $(\tilde u,\tilde v,\tilde w)$ satisfies
\begin{equation}\label{LE-hyp}
\displaystyle\limsup_{t\to +\infty} \left(\zeta \tilde u(t,ct)+ \xi \tilde v(t,ct)\right)\geq \varepsilon(c).
\end{equation}
Then for any $c\in [c_2,c_1)$ the solution $(u,v,w)$ satisfies:
\begin{equation}\label{inf}
\liminf_{t\to + \infty}\inf_{c_2t \leq x\leq c t}\left(\zeta u(t,x)+ \xi v(t,x)\right)>0,
\end{equation}
as well as
\begin{equation}\label{inf_w}
\limsup_{t\to +\infty}\sup_{c_2t \leq x\leq c t} w(t,x) < 1.
\end{equation}
\end{lemma}

\begin{proof}
The proof of this lemma involves three main steps described below.

\noindent{\bf First step.} In this first step we shall show that for any $c\in [c_2,c_1)$ there exists $\varepsilon_1(c)>0$ such that
\begin{equation}\label{first}
\liminf_{t\to +\infty} \left(\zeta u(t,ct)+ \xi v(t,ct)\right)\geq \varepsilon_1(c).
\end{equation}
To prove the above lower estimate we argue by contradiction by fixing $c\in  [c_2,c_1)$ and by assuming (using~\eqref{LE-hyp} for $(u,v,w)$) that there exist a sequence $t_n\to +\infty$ and a sequence $s_n> t_n$
such that for all $n$ large enough
\begin{equation*}
\begin{cases}
[\zeta u+ \xi v](t_n,c_nt_n)=\frac{\varepsilon(c)}{2},\\
[\zeta u+ \xi v](t,c_nt)\leq \frac{\varepsilon(c)}{2},\;\forall t\in [t_n,s_n],\\
[\zeta u+ \xi v](s_n,c_ns_n)\leq \frac{1}{n}.
\end{cases}
\end{equation*}
Then, possibly up to a sub-sequence not relabelled, one may assume that
\begin{equation*}
(u,v,w)(t+t_n,x+c t_n )\to 
(u_\infty,v_\infty,w_\infty)(t,x),
\end{equation*}
as $n \to +\infty$, where the limit functions $(u_\infty, v_\infty , w_\infty)$ solve \eqref{u-eq}-\eqref{w-eq} for all $t \in \R$ and
$$(u_\infty , v_\infty ,  w_\infty ) (0, \cdot) \in \omega_0 (c_2,c_1).$$
We also have by construction that
\begin{equation*}
\zeta  u_\infty (0,0)  + \xi  v_\infty (0,0)=\frac{\varepsilon(c)}{2},
\end{equation*}
so that, by the strong maximum principle, $\zeta u_\infty+ \xi  v_\infty>0$.
Furthermore let us notice that $s_n-t_n\to +\infty$ as $n\to +\infty$. Indeed, if the sequence $\{s_n-t_n\}$ has a converging sub-sequence to $\sigma\in\R$, then the limit functions satisfy
\begin{equation*}
\zeta  u_\infty (\sigma,c\sigma)+ \xi  v_\infty (\sigma,c\sigma)=0.
\end{equation*}
So that if for instance $\zeta  u_\infty (0,0)>0$ -- hence $\zeta>0$ -- then the previous equality ensures that $ u_\infty (\sigma,c\sigma)=0$ and $ u_\infty (t,x)=0$ for any $(t,x)\in\R^2$, a contradiction.

As a consequence, since $s_n-t_n\to +\infty$ as $n\to +\infty$, one obtains that
\begin{equation*}
\begin{split}
&\zeta  u_\infty (t,ct)+ \xi  v_\infty (t,ct)\leq \frac{\varepsilon(c)}{2},\;\forall t\geq 0.
\end{split}
\end{equation*}
Recalling that $( u_\infty ,  v_\infty,  w_\infty  ) (0,\cdot) \in \omega_0 (c_2,c_1)$ and $\zeta  u_\infty (0,\cdot) + \xi  v_\infty (0,\cdot) \not \equiv 0$, this contradicts~\eqref{LE-hyp}
and completes the proof of \eqref{first}.

\noindent{\bf Second step.} We now prove \eqref{inf}. To that aim we consider the shifted function
$(\hat u,\hat v,\hat w)(t,x):=(u,v,w)(t,x+c_2t)$.
Once more, we proceed by contradiction. Let $\tilde c\in(c_2,c_1)$ be given. Assume that there exist $t_n \to +\infty$, $c_n \in [0,\tilde c-c_2)$ such that
\begin{equation*}
\lim_{n\to +\infty} (\zeta u+ \xi v)(t_n, (c_2+c_n) t_n)=\lim_{n\to +\infty} (\zeta \hat u+ \xi \hat v)(t_n, c_n t_n)=0.
\end{equation*}
Without loss of generality, up to a sub-sequence, we assume that $c_n \to c\in [0,\tilde c-c_2]$. Choose $c'$ such that $c<c'<c_1-c_2$
and define the sequence
$$
t'_n := \frac{c_n t_n}{c'}\in [0,t_n),\;\forall n\geq 0.
$$
Consider first the case when the sequence $\{c_n t_n \}$ is bounded which may happen if $c=0$. Then up to extraction of a sub-sequence, one has as $n \to +\infty$ that
$$c_n t_n \to x_\infty \in \R,$$
and, due to the strong maximum principle,
$$
\lim_{n\to +\infty}(\zeta \hat u+\xi \hat v)(t_n +t,c_n t_n + x)=0\text{ locally uniformly for $(t,x)\in\R^2$}.
$$
This implies in particular that $(\zeta \hat u+\xi \hat v)(t_n,0)=(\zeta u+ \xi v)(t_n,c_2t_n) \to 0$ as $n \to +\infty$, which contradicts \eqref{first} with $c=c_2$.
As a consequence $c>0$, the sequence $\{c_nt_n\}$ has no bounded sub-sequence and we can assume below that $t_n'\to +\infty$ as $n\to +\infty$.

Now due to \eqref{first} we have for all large $n$ that
$$(\zeta \hat u+\xi \hat v) (t_n ', c_nt_n) = (\zeta\hat u+\xi \hat v) (t_n ', c' t_n ')=(\zeta u+\xi v) (t_n ', (c_2+c') t_n ')> \frac{3}{4}\varepsilon_1( c_2+c').$$
Next, we introduce a third time sequence $\{ t''_n \}$ by
$$ t''_n := \inf \left\{ t \leq  t_n \, | \ \forall s \in (t , t_n), \, (\zeta \hat u+\xi \hat v) (s, c_n t_n)  \leq \frac{\min \{\varepsilon(c_2), \varepsilon_1(c_2+c')\}}{2} \right\} \in (t' _n, t_n).$$
Since $(\zeta \hat u+\xi \hat v)(t_n , c_n t_n) \to 0$ as $n \to +\infty$, we get
$$(\zeta \hat u+\xi \hat v) (t_n '' , c_n t_n) = \frac{\min \{\varepsilon(c_2), \varepsilon_1( c_2+c')\}}{2},$$
and, as before, by a limiting argument and a strong maximum principle, that
$$t_n - t_n '' \to +\infty$$
as $n \to +\infty$.

Then, by parabolic estimates and up to extraction of a sub-sequence, we find that $(\hat u,\hat v,\hat w)(t + t_n '' ,x + c_nt_n - c_2 t)$ converges to a solution $(u_\infty, v_\infty , w_\infty)$ such that
\begin{equation}\label{pm}
\begin{split}
&(\zeta u_\infty+\xi v_\infty)(0,0)>0,\\
&(\zeta u_\infty+\xi v_\infty)(t,c_2 t)\leq  \frac{\min \{\varepsilon(c_2), \varepsilon_1( c_2+c')\}}{2},\;\forall t\geq 0.
\end{split}
\end{equation}
Note  also that $(\hat u,\hat v,\hat w)(t + t_n '' ,x + c_nt_n - c_2 t)=(u,v,w)(t+t_n'', x+c_nt_n+c_2t_n'')$. Hence, since $0\leq c_nt_n=c't_n'\leq c't_n''\leq (c_1-c_2)t_n''$, one gets
$$
x_n:=c_nt_n+c_2t_n''\in [ c_2 t_n'', c_1 t_n'' ],
$$
so that
$(u_\infty,v_\infty,w_\infty)(0,\cdot)\in \omega_0 (c_2,c_1)$ and \eqref{pm} contradicts~\eqref{LE-hyp}. This completes the proof of \eqref{inf}.

\noindent{\bf Third step.} Finally we deal with \eqref{inf_w} and the $w$ component. We again proceed by contradiction.
Let $c\in [c_2,c_1)$ be given and assume that $w (t_n, x_n) \to 1$ for some sequences $t_n \to +\infty$ and $c_2t_n \leq x_n \leq c t_n$.
Then, up to extraction of a sub-sequence, $(u,v,w) (t+t_n , x + x_n)$ converges to an entire in time (i.e., for all $t \in \R$) solution $(u_\infty, v_\infty, w_\infty)$ of \eqref{u-eq}-\eqref{w-eq}, which also satisfies
$$w_\infty (0,0) = 1 .$$
It follows from the strong maximum principle that $w_\infty \equiv 1$. However, according to the above we have $\zeta u_\infty + \xi v_\infty >0$, hence either $u_\infty >0$ or $v_\infty>0$.
This is a contradiction to \eqref{w-eq} and the lemma is proved.
\end{proof}

The above lemma can be strengthened as follows in the coupling case $\mu>0$.
\begin{corollary}\label{COR-weak-strong}
Assume that $\mu>0$ and $(u_0,v_0,w_0)\in X_0$. Let $c_1 > 0 $ be given. 
Assume also that for any $c\in [0,c_1)$ there exists $\varepsilon(c)>0$ such that for any $(\tilde u_0,\tilde v_0,\tilde w_0)\in \omega_0(0,c_1)$ with $\tilde u_0+\tilde v_0\not\equiv 0$,
the corresponding solution $(\tilde u, \tilde v, \tilde w)$ satisfies
\begin{equation*}
\limsup_{t\to +\infty}\,\left( \tilde u(t,ct)+\tilde v(t,ct)\right) \geq \varepsilon(c).
\end{equation*}
Then for any $c\in [0,c_1)$ the solution $(u,v,w)$ satisfies
\begin{equation}\label{inf-uv}
\liminf_{t\to +\infty}\inf_{0\leq x\leq c t}\min\left\{u(t,x), v(t,x), 1 - w(t,x) \right\}>0.
\end{equation}
\end{corollary}
The proof of this result follows from a straightforward limiting argument together with the strong maximum principle.
 Indeed, for any $c\in[0,c_1)$, from Lemma~\ref{LE-weak-strong} it follows that
\begin{equation}\label{inf-uv0}
\liminf_{t\to +\infty}\inf_{0\leq x\leq c t}\min\left\{u(t,x)+v(t,x), 1 - w(t,x) \right\}>0.
\end{equation}
Next, as in the third step of the proof of Lemma~\ref{LE-weak-strong}, if there exist sequences $t_n \to +\infty$ and $0 \leq x_n \leq c t_n$ such that $u (t_n,x_n) \to 0$ as $n \to +\infty$,
then $(u,v,w) (t+t_n,x+x_n)$ converges to an entire in time solution $(0,v_\infty,w_\infty)$. From~\eqref{inf-uv0} we must have $v_\infty >0$, which contradicts~\eqref{u-eq}. It follows that
$$\liminf_{t \to +\infty} \inf_{0 \leq x \leq ct } u(t,x) >0 .$$
Proceeding similarly for the other predator, one reaches the conclusion \eqref{inf-uv}.

\begin{remark}
The same results as in Lemma \ref{LE-weak-strong} and Corollary \ref{COR-weak-strong} also hold for negative speeds with straightforward adaptations.
\end{remark}

\subsection{Some eigenvalue problems}

We conclude this section by a computational lemma that shall be used in the sequel in particular to check the assumptions of Lemma~\ref{LE-weak-strong} and Corollary \ref{COR-weak-strong}.
\begin{lemma}\label{LE-eig}
Let $d>0$, $c\in\R$, $a\in\R$ and $R>0$ be given. Then the principal eigenvalue $\lambda_R$ of the following Dirichlet elliptic problem
\begin{equation*}
\begin{cases}
-d\varphi''(x)- c\varphi'(x)+a\varphi(x)=\lambda_R\varphi(x),\; \mbox{ for }x \in (-R,R), \vspace{3pt}\\
\varphi(\pm R)=0\text{ and }\varphi>0\text{ on }(-R,R),
\end{cases}
\end{equation*}
is given by
\begin{equation*}
\lambda_R=a+\frac{c^2}{4d}+\frac{d\pi^2}{4R^2} .
\end{equation*}
\end{lemma}
This is classical and we omit the proof.

In the `two mutants' case, we need an analogous result for a (cooperative) system as follows. 
\begin{proposition}\label{prop:linear_instab_mutant}
Let $\mu>0$, $c \in \R$, $\delta\geq 0$ and $R >0$ be given. Consider the following principal eigenvalue problem:
\begin{equation*}
\left\{
\begin{array}{l}
-d_1 \varphi_{xx} - c \varphi_x - r_1 \varphi  (a-1 -2 \delta )  - \mu (\psi - \varphi) = \Lambda_R \varphi, \ \ \mbox{ for } x \in (-R,R), \vspace{3pt}\\
- d_2 \psi_{xx} - c \psi_x - r_2 \psi (a-1 - 2 \frac{r_1}{r_2}\delta  )  - \mu (\varphi - \psi )= \Lambda_R \psi , \ \ \mbox{ for } x \in (-R,R) , \vspace{3pt}\\
\varphi (\pm R) = \psi (\pm R) = 0  \text{ and } \varphi ,\psi >0 \text{ on } (-R,R).
\end{array}
\right.
\end{equation*}
Then there exists a unique $\Lambda_R(c,\delta)$ such that this eigenvalue problem admits an eigenfunction pair $(\varphi, \psi)$ with $\varphi, \psi >0$ in $(-R,R)$.
It satisfies $\Lambda_R(c,\delta)=\Lambda_R(c,0) + 2r_1\delta$ and $\Lambda_R(-c,0)=\Lambda_R(c,0)$.
Moreover, for each $|c|<c^*_\mu$, there exist $\delta_0>0$ and $R_0$ such that
\begin{equation*}
\Lambda_R(c,\delta) < 0,\;\forall \delta\in [0,\delta_0],\;\forall R\geq R_0.
\end{equation*}
\end{proposition}
Since this result was proved in~\cite{Girardin18}, we only give a brief sketch of the argument.
\begin{proof}
First, by the classical Krein-Rutman theory, there exists a unique $\Lambda_R(c,\delta)$ such that this eigenvalue problem admits a positive eigenfunction pair.
The symmetry with respect to $c$ follows from the uniqueness of this eigenvalue and by changing $x$ to $-x$ into the system.

We consider now $0\leq c<c^*_\mu$.
As $R \to +\infty$, $\Lambda_R(c,\delta)$ converges to a generalized principal eigenvalue of the operator
\begin{equation*}
\mathbf{A} [\varphi, \psi ] :=
\left(
\begin{array}{l}
-d_1 \varphi_{xx} -c \varphi_x - r_1 \varphi  (a-1-2\delta) -  \mu (\psi - \varphi)\\
- d_2 \psi_{xx} - c \psi_x - r_2 \psi (a-1-2\frac{r_1}{r_2}\delta)  - \mu (\varphi - \psi )
\end{array}
\right);
\end{equation*}
see~\cite[Theorem 4.2]{Girardin18}. More precisely, this generalized principal eigenvalue is defined as
$$\sup \, \{ \lambda \in \R \, | \ \exists (\varphi,\psi) \in C^2 (\R, \R_+^* \times \R_+^*) , \quad  \mathbf{A} [\varphi, \psi] \geq \lambda [\varphi ,\psi] \},$$
where the inequality is to be understood componentwise. It also turns out (see~\cite[Lemma 6.4]{Girardin18}) that this generalized principal eigenvalue coincides with
the maximum of the function $$\gamma \in [0, +\infty) \mapsto - \Lambda [\mu,\gamma] + 2 r_1 \delta + c\gamma ,$$
where $\Lambda [\mu,\gamma]$ is the Perron-Frobenius eigenvalue of the matrix defined by~\eqref{matrix_mutant}. In other words, one has uniformly with respect to $\delta$ that
$$\lim_{R \to +\infty} \Lambda_R(c,\delta) = \max_{\gamma \geq 0}  \left\{ -\Lambda [\mu,\gamma] + c\gamma \right\}  + 2r_1\delta.$$
Recalling that $\Lambda [\mu,0] >0$,
$$ c_\mu^* := \min_{\gamma >0} \frac{ \Lambda [\mu,\gamma] }{\gamma} >0 ,$$
and $0\leq c<c^*_\mu$, one gets the existence of $R_0>0$ large enough and $\delta_0>0$ small enough such that for all $R\geq R_0$ and $\delta\in [0,\delta_0]$ one has
$$
\Lambda_R(c,\delta) < 0.
$$
The proposition is proved.
\end{proof}

We point out that it follows from this proposition that solutions of the sub-system
\begin{equation*}
\left\{
\begin{array}{l}
u_t = d_1 u_{xx} + r_1 u (a-1 -  u -k v) + \mu (v-u), \vspace{3pt}\\
v_t = d_2 v_{xx} + r_2 v (a-1 -  h u - v) + \mu (u-v),
\end{array}
\right.
\end{equation*}
which arises from taking $w \equiv 1$ in \eqref{u-eq}-\eqref{v-eq}, spread with speed $c^*_\mu$ when $0 < \mu \leq (a-1) \min \{r_1,r_2\}/2$.
%
We refer to Girardin~\cite{Girardin18} for a proof of this result, in a more general framework including an arbitrary number of mutants.
This relies on the fact that this sub-system roughly has a cooperative structure around the trivial steady state $(0,0)$, i.e., the linearized system around $(0,0)$ is cooperative, since competitive terms are nonlinear.
As a matter of fact, we believe that our analysis for the predator-prey system could also work when we increase the number of mutants.
As we shall see later, our proof relies on a similar construction of subsolutions as in~\cite{Girardin18}, hence on the above proposition.

\section{Uniformly positive entire solutions}\label{sec:lyapunov}

In this section, we state some Liouville type results on the stationary solutions of both the full `two competitors' system and its sub-systems with only one predator.
This relies on a Lyapunov approach and it shall allow us to describe the shape of the solutions behind the propagation fronts. Throughout this section we shall assume that
$$\mu = 0.$$

\subsection{Two-dimensional sub-systems}

We start with an important lemma on bounded and uniformly positive entire solutions of the sub-system
\begin{equation*}
\begin{cases}
u_t -d_1 u_{xx}=u F(u,0,w) ,\\
w_t-d_3 w_{xx}=w H(u,0,w),
\end{cases} t\in\R,\;x\in\R.
\end{equation*}
Our result reads as follows.

\begin{lemma}\label{LE_entire}
Let $(u,w)=(u,w)(t,x)$ be an entire solution of the above system such that there exist constants $M>m>0$ with
$$
m\leq u(t,x)\leq M,\;m\leq w (t,x)\leq M,\;\forall (t,x)\in\R^2.
$$
Then $(u,w)\equiv \left(\tilde p,\tilde q\right)$, where the stationary state $\left(\tilde p,\tilde q\right)$ is defined in \eqref{def-pq}.
\end{lemma}
Before we prove this result, we point out that the roles of $u$ can easily be inverted. In particular, considering
\begin{equation*}
\begin{cases}
v_t-d_2 v_{xx} =v G(0,v,w) , \\
w_t  -d_3 w_{xx} =w H(0,v,w),
\end{cases} t\in\R,\;x\in\R ,
\end{equation*}
we also have:
\begin{lemma}\label{LE_entire_bis}
Let $(v,w)=(v,w)(t,x)$ be an entire solution of the above system such that there exist constants $M>m>0$ with
$$
m\leq v(t,x)\leq M,\;m\leq w(t,x)\leq M,\;\forall (t,x)\in\R^2.
$$
Then $(v,w)\equiv \left(\tilde p,\tilde q\right)$, where the stationary state $\left(\tilde p,\tilde q\right)$ is defined in \eqref{def-pq}.
\end{lemma}

\begin{proof}
Since both results are actually the same up to renaming parameters, we only prove Lemma~\ref{LE_entire}.
In order to prove this lemma we make use of a Lyapunov like argument, inspired from \cite{DH04,DG18}.

Define the function $g(x)=x-\ln(x)- 1$. Recalling the definition of $\left(\tilde p,\tilde q\right)$ in \eqref{def-pq} one has
$$
F(u,0,w)=r_1[\tilde p-u+a(w-\tilde q)],\;\;H(u,0,w)=r_2[b(\tilde p-u)+(\tilde q-w)].
$$
Consider also the function $V=V(u,w)$ given by
\begin{equation*}
V(u,w)=br_2\tilde p g\left(\frac{u}{\tilde p}\right)+ar_1\tilde q g\left(\frac{w}{\tilde q}\right):=V_1(u)+V_2(w).
\end{equation*}
Note that the Lie Derivative of $V$ along $X=(uF(u,0,w), wH(u,0,w))$, denoted by $L_X V$, is given by
\begin{equation*}
\begin{split}
L_X V(u,w)&=V_u(u,w) uF(u,0,w)+V_w(u,w)w H(u,0,w)\\
&=br_1 r_2 \left(u-\tilde p\right)[\tilde p-u+a(w-\tilde q)]+ar_1r_2(w-\tilde q)[b(\tilde p-u)+(\tilde q-w)]\\
&=-br_1 r_2 \left(u-\tilde p\right)^2-ar_1r_2(w-\tilde q)^2.
\end{split}
\end{equation*}
Now observe that there exists some constant $\alpha>0$  such that
$$
L_X V(u,w)\leq -\alpha V(u,w),\;\forall (u,w)\in [m,M]\times [m,M].
$$

Next, consider a function $\varphi$ smooth, non-negative and such that
$$
\varphi(x)=\begin{cases} 1,\quad\text{ if $x\in [-1,1]$,}\\ 0,\quad\text{ if $|x|\geq 2$.}\end{cases}
$$
Define for $R>0$ the functional
\begin{equation*}
 \mathcal E_R (t) =\int_\R \varphi(R^{-1}x) V(u(t,x),w(t,x))dx.
\end{equation*}
Then one has, for all $t\in\R$,
\begin{equation*}
\frac{d}{dt} \mathcal E_R(t)=\int_{\R} \varphi(R^{-1}x) [V_1'(u)d_1 u_{xx}+V_2'(w)d_3w_{xx}] dx+\int_{\R}\varphi(R^{-1}x)L_X V(u,w)dx ,
\end{equation*}
so that
\begin{equation*}
\begin{split}
\frac{d}{dt} \mathcal E_R(t)=& R^{-2}\int_{\R} \varphi''(R^{-1}x) d_1 V_1(u)dx -d_1\int_{\R} \varphi(R^{-1}x)V_1''(u)(u_x)^2 dx\\
&+R^{-2}\int_{\R} \varphi''(R^{-1}x) d_3 V_2(u) dx-d_3\int_{\R} \varphi(R^{-1}x)V_2''(u)(w_x)^2 dx\\
&+\int_{\R}\varphi(R^{-1}x)L_X V(u,w)dx.
\end{split}
\end{equation*}
Since $V_1''\geq 0$, $V_2''\geq 0$ and $ m \leq u,v \leq M$, we obtain that there exists some constant $K>0$ independent of $R$ such that
\begin{equation*}
\frac{d}{dt} \mathcal E_R(t)\leq \frac{K}{R}-\alpha \mathcal E_R(t),\;\forall t\in\R.
\end{equation*}
This yields
$$
\mathcal E_R(t)\leq \frac{K}{\alpha R},\;\forall t\in\R,\;\forall R>0.
$$
This ensures that $\mathcal E_R(t)\to 0$ as $R\to +\infty$ uniformly for $t\in\R$, and the result follows by applying Fatou's lemma.
\end{proof}

\subsection{The full system}

\begin{lemma}\label{LE-entire-full}
Let $(u,v,w) = (u,v,w) (t,x)$ be a bounded entire solution of \eqref{u-eq}-\eqref{w-eq} such that
\begin{equation}\label{lb}
\min\left(\inf_{(t,x)\in\R^2} u(t,x), \inf_{(t,x)\in\R^2}v(t,x),\inf_{(t,x)\in\R^2}w(t,x)\right) >0.
\end{equation}
Then
$
(u,v,w) \equiv \left(u^*,v^*,w^*\right),
$
where the stationary state $(u^*,v^*,w^*)$ is defined in~\eqref{def-equi}.
\end{lemma}

\begin{proof}
The proof of this lemma makes use of similar arguments as the ones used for Lemma \ref{LE_entire}.
Let $(u,v,w)$ be a bounded entire solution of \eqref{u-eq}-\eqref{w-eq} satisfying \eqref{lb}.
We denote by $0<m\leq M$ the lower and upper bounds of $(u,v,w)$.
Recall that we defined the function $g (x)= x - \ln (x) -1$, and let us consider the function $\Phi=\Phi(u,v,w)$ given by
$$
\Phi(u,v,w):=u^*g\left(\frac{u}{u^*}\right)+\frac{r_{1}v^*}{r_{2}}g\left(\frac{v}{v^*}\right)+\frac{r_{1} aw^*}{r_{3} b}g\left(\frac{w}{w^*}\right).
$$

Let us compute the Lie derivative of $\Phi$, denoted by $L_X \Phi$, along the three dimensional vector field $X=(uF(u,v,w), vG(u,v,w), wH(u,v,w))$ associated to the kinetic part of \eqref{u-eq}-\eqref{w-eq}.
It reads for each $u>0$, $v>0$ and $w>0$ as:
 \begin{align*}
&L_X\Phi( u,v,w)\\
&=
r_1 \{ -(u-u^*)^{2}-k(u-u^*)(v-v^* )+a(u-u^*)(w-w^*) \}\\
&\qquad +r_1 \{ -h(u-u^*)(v-v^*) -(v-v^*)^{2}+a(v-v^*)(w-w^*) \}\\
&\qquad + \frac{r_{1}a}{b} \{ -b(u-u^*)(w-w^*)-b(v-v^*)(w-w^*)-(w-w^*)^{2} \}
\\
&=
-r_{1}(u-u^*)^{2}
-r_{1}(k+h)(u-u^{*})(v-v^{*})-r_{1} (v-v^{*})^{2}-\frac{r_{1}a}{b}(w-w^{*})^{2}.
\end{align*}
Now observe that for all $(X_1,X_2)=(r\cos\theta,r\sin\theta)\in\R^2$ one has
\begin{equation*}
\begin{split}
&X_1^2+(k+h)X_1X_2+X_2^2= r^2 (1+(h+k)\cos\theta \sin\theta)\\
&=r^2\left[1+\frac{h+k}{2}\sin(2\theta)\right]\geq r^2\left[1-\frac{h+k}{2}\right]=\left[1-\frac{h+k}{2}\right](X_1^2+X_2^2).
\end{split}
\end{equation*}
As a consequence, since $0<h,k<1$ (see \eqref{c1}) one has $0<h+k<2$ and there exists $\alpha>0$ such that for all $u>0$, $v>0$ and $w>0$ one has
\begin{equation*}
L_X\Phi(u,v,w)\leq -\alpha\left[(u-u^{*})^2+(v-v^{*})^2+(w-w^{*})^{2}\right].
\end{equation*}
 Furthermore, recalling the definition of $\Phi$ above, there exists $\beta>0$ such that
 \begin{equation*}
 L_X\Phi( u,v,w)\leq -\beta \Phi( u,v,w),\;\forall (u,v,w)\in [m,M]^3.
 \end{equation*}
 From this inequality, the proof of Lemma \ref{LE-entire-full} follows from the same arguments as the ones used for Lemma~\ref{LE_entire}.
\end{proof}

\section{Proofs of Theorem \ref{th:mutants} and Theorem \ref{theo:spread_plus}}\label{sec:mutants}

In this section, we mostly deal with the `two mutants' case, i.e., Theorem~\ref{th:mutants}. However, because the arguments for the upper bound on the speeds is very similar,
we also include the proof of Theorem~\ref{theo:spread_plus}.

\subsection{Upper bounds on the spreading speeds}

In this subsection we derive upper estimates for the spreading speed of the solution of system \eqref{u-eq}-\eqref{w-eq}. Our first result is given below.
It proves both part~(i) of Theorem \ref{th:mutants} and Theorem \ref{theo:spread_plus}.

\begin{theorem}\label{theo:spread_plusplus}
Let $(u_0 ,v_0, w_0) \in X_0$ (recall \eqref{X0}) be such that $u_0$ and $v_0$ are both nontrivial and compactly supported, and $(u,v,w)$ be the corresponding solution.
\begin{itemize}
\item[(i)] If $ \mu = 0$, then
\be\label{u-beyond}
\lim_{ t \to +\infty} \sup_{|x| \geq ct} u(t,x) =0,
\ee
for any $c>c_u^*$, and
\be\label{v-beyond}
\lim_{ t \to +\infty} \sup_{|x| \geq ct} v(t,x) =0,
\ee
for any $c>c_v^*$; moreover,
\be\label{w-beyond}
\lim_{ t \to +\infty} \sup_{|x| \geq ct} | w (t,x) -1| = 0,
\ee
for any $c > \max \{ c_u^* ,c_v^* \}$.
\item[(ii)] If $\mu > 0$, then
\bea
&&\lim_{ t \to +\infty} \sup_{|x| \geq ct} u(t,x) =0,\label{u-beyond-mu}\\
&& \lim_{ t \to +\infty} \sup_{|x| \geq ct} v(t,x) =0 , \label{v-beyond-mu}\\
&&\lim_{ t \to +\infty} \sup_{|x| \geq ct} | w (t,x) -1| = 0,\label{w-beyond-mu}
\eea
for any $c >  c_\mu^*$.
\end{itemize}
\end{theorem}
Notice that in the case when $\mu >0$, then one can assume without loss of generality and as in Theorem~\ref{th:mutants} that the pair $(u_0,v_0)$ is nontrivial.
This immediately comes from the fact that both components must be simultaneously positive or simultaneously identical to 0 for all positive times.

\begin{proof}
First, due to $w\le 1$ and $u,v\ge 0$ we have that
\beaa
\begin{cases}u_t \leq d_1 u_{xx} + r_1 u (a-1 - u) + \mu (v-u),\\
v_t \leq d_2 v_{xx} + r_2 v (a-1-v) + \mu (u-v),
\end{cases}
\eeaa
for all $t >0 $ and $x \in \R$.
In the case when $\mu = 0$, this system is actually uncoupled so that \eqref{u-beyond} and \eqref{v-beyond} follow from a comparison principle and the classical result of~\cite{AW75} for the scalar equation.

In the case when $\mu >0$, we further notice that $u,v$ is a subsolution of the following linear system,
\beaa
\begin{cases}
u_t \leq d_1 u_{xx} + r_1 u (a-1 ) + \mu (v-u),\\
v_t \leq d_2 v_{xx} + r_2 v (a-1) + \mu (u-v),
\end{cases}
\eeaa
which satisfies a comparison principle. Recalling the definition of $c_\mu^*$ in~\eqref{cmu*}, this linear system admits a solution of the type
$$(p,q) e^{-\gamma (x-c_\mu^* t)},$$
where $p$, $q$ and $\gamma$ are positive. By applying the comparison principle, we obtain \eqref{u-beyond-mu} and \eqref{v-beyond-mu}.

Let us now deal with $w$, i.e., \eqref{w-beyond} and~\eqref{w-beyond-mu}. The argument is the same in both cases.
Proceed by contradiction and assume that there exists some $c> \max \{ c_u^* ,c_v^* \}$ (if $\mu =0$) or $c > c_\mu^*$ (if $\mu >0$), and a sequence $\{(t_n, x_n)\}$ such that $t_n \to +\infty$, $|x_n| \geq c t_n$ and
$$\limsup_{n \to +\infty}  w(t_n,x_n) < 1.$$
Thanks to standard parabolic estimates (recall that all three components are uniformly bounded by Proposition~\ref{prop:prelim_estim}), we can assume up to extraction of a sub-sequence that $(u,v,w)(t+t_n,x+x_n)$
converges to an entire in time solution $(u_\infty , v_\infty, w_\infty)$ of the same system.
From what we have proved above, we know that $u_\infty \equiv v_\infty \equiv 0$, and thus~$w_\infty$ solves
$$(w_\infty)_t = d_3 (w_\infty)_{xx} + r_3 w_\infty (1 - w_\infty).$$
Recalling that $w$, hence $w_\infty$, is everywhere larger than $\beta$, we conclude that $w_\infty \equiv 1$. This is a contradiction and completes the proof.
\end{proof}

\subsection{Spreading for the two mutants system: the case when $\mu >0$}

In this subsection, we assume {that $0 < \mu \leq (a-1) \min \{ r_1, r_2 \}/2$,} so that $u$ and $v$ denote the densities of two mutant types of the same predator species.
As we mentioned before, due to $\mu > 0$, the sub-system composed of equations~\eqref{u-eq}-\eqref{v-eq} with constant $w$ roughly has a cooperative structure around the trivial steady state $(0,0)$.

To prove the second part of Theorem~\ref{th:mutants}, we apply Lemma \ref{LE-weak-strong} and more precisely its Corollary \ref{COR-weak-strong} with $c_1 = c_\mu^*$.
According to this aforementioned result, part~(ii) of Theorem~\ref{th:mutants} directly follows from the next lemma, which shows that $u$ and $v$ cannot go extinct simultaneously in a moving frame with speed less than $c^*_\mu$.

\begin{lemma}\label{weak_sum}
For any $c \in [0, c^*_\mu)$, there exists some $\varepsilon (c)>0$ such that, for any initial data satisfying $u_0 + v_0 \not \equiv 0$, $0 \leq  u_0 ,v_0 \leq a-1$ and $\beta \leq w_0 \leq 1$,
the corresponding solution $(u,v,w)$ satisfies
$$\limsup_{t \to +\infty} \, (u(t,ct) + v (t,ct) ) \geq \varepsilon (c).$$
\end{lemma}
\begin{proof}
First notice that, because $\mu >0$, and even if one of the two functions $u_0$ and $v_0$ is identically equal to 0, the strong maximum principle ensures that both $u >0$ and $v >0$ for positive times.

Let $c \in [0, c^*_\mu)$ and assume by contradiction that there is a sequence of solutions $\{(u_n,v_n , w_n)\}$ with initial data $\{(u_{0,n}, v_{0,n}, w_{0,n})\}$ such that
\begin{equation}\label{eq:weak_sum1}
\limsup_{t \to +\infty}  \, (u_n (t,ct) + v_n (t,ct) ) \leq \frac{1}{n}.
\end{equation}
This clearly implies that there exists $t_n$ large enough such that
$$ \max \, \{ u_n (t, ct) , v_n (t,ct) \} \leq \frac{2}{n} , \ \forall t \geq t_n.$$
In particular, passing to the limit as $n \to +\infty$, and applying a strong maximum principle, one may check that for any $R>0$,
$$\limsup_{n \to +\infty} \sup_{ t \geq t_n, |x-ct| \leq R }  ( u_n (t,x) + v_n (t,x) ) =  0.$$
We next claim that, for any $R>0$,
\be\label{claim_weakspread_w1}
\limsup_{n \to +\infty} \sup_{t \geq t_n, |x-ct | \leq R} |w_n (t,x) - 1| =0.
\ee
Proceed by contradiction and take a sequence $\{(s_n,x_n)\}$ with $s_n \geq t_n$ and $x_n \in (cs_n -R,cs_n+R)$ such that $\limsup_{n \to +\infty} w_n (s_n,x_n) < 1$.
Since solutions are bounded uniformly with respect to $n$,
we can use standard parabolic estimates and extract a sub-sequence so that $(u_n, v_n, w_n)(t+s_n,x+x_n)$ converges to an entire in time solution $(u_\infty, v_\infty, w_\infty)$.
By construction, $u_\infty, v_\infty \geq 0$ and $u_\infty (t,0) = v_\infty (t,0)= 0$ for all $t >0$. Hence, by the strong comparison principle, we have that $u_\infty \equiv v_\infty \equiv 0$.
Thus $w_\infty$ satisfies
$$(w_\infty)_t = d_3 (w_\infty)_{xx} + r_3 w_\infty (1 - w_\infty).$$
Since $w_\infty \geq \beta$, we deduce that $w_\infty \equiv 1$, which contradicts the fact that $w_\infty (0,0) < 1$. The claim \eqref{claim_weakspread_w1} is thus proved.

Then, for any small $\delta >0$ and large $R>0$, one can increase $n$ so that $(u_n, v_n)$ satisfies
\begin{equation*}
\left\{
\begin{array}{l}
(u_n)_t \geq d_1 (u_n)_{xx} + r_1 u_n  (a-1 -2 \delta ) + \mu (v_n -u_n), \vspace{3pt}\\
(v_n)_t \geq d_2 (v_n)_{xx} + r_2 v_n (a-1 - 2 \frac{r_1}{r_2} \delta   )  + \mu (u_n-v_n) ,
\end{array}
\right.
\end{equation*}
for $t \geq t_n$ and $|x -ct_n| \leq R$. Notice that this is a cooperative system, hence it satisfies the comparison principle. In particular, we get that
$$u_n  (t, x +ct) \geq \underline{u}_n (t,x) \, , \quad v_n (t, x+ ct) \geq \underline{v}_n (t,x) ,$$
for all $t \geq t_n$ and $|x| \leq R$, where
$$(\underline{u}_n ,\underline{v}_n )(t,x) := \epsilon e^{-\Lambda_R t} (\varphi , \psi) (x),$$
with $\Lambda_R$, $(\varphi,\psi)$ the principal eigenvalue and the associated positive eigenfunction pair from Proposition~\ref{prop:linear_instab_mutant}, and $\epsilon >0$ is small enough (possibly depending on $n$) so that
$$u_n (t_n , x + c t_n) \geq \epsilon e^{-\Lambda_R t_n} \varphi (x) \, , \quad v_n (t_n, x + ct_n) \geq \epsilon e^{-\Lambda_R t_n} \psi (x),$$
for all $|x| \leq R$. Using again Proposition~\ref{prop:linear_instab_mutant}, we have that $\Lambda_R < 0$ if $\delta$ is small enough and $R$ is large.
Thus, $u_n (t,ct) \geq \underline{u}_n (t,ct) \to +\infty$ as $t \to +\infty$, which contradicts \eqref{eq:weak_sum1}. The lemma is proved.
\end{proof}
Recall here that the above lemma completes the proof of part~(ii) of Theorem~\ref{th:mutants}, by applying Corollary \ref{COR-weak-strong} and a symmetry argument to deal with the negative part of the spatial domain.
Together with Theorem~\ref{theo:spread_plusplus}, this ends the proof of Theorem~\ref{th:mutants}.

\section{Lower bounds on the speeds in the competitor case}\label{sec:competitors1}

In this section we derive preliminary lower spreading estimates for the two competitors system, namely when $\mu=0$. More precisely, we prove parts~(i) and~(ii) of Theorem~\ref{THEO-inside1} in the case when $c_u^* = c_v^*$.
However, the results proved in this section can also be applied in the case when $c_v^* < c_u^*$ and they shall serve as a starting point in the proofs of Theorem~\ref{inside} in Section~\ref{competitors_faster}.

Throughout this section we assume that $\mu = 0$ and
$$
c_v^*\leq c_u^*.
$$
Along this section the initial data is {\bf not} assumed to be compactly supported but it is a general function $(u_0,v_0,w_0)\in X_0$, where $X_0$ has been defined in~\eqref{X0}.
In particular, we have that $0 \leq u_0 , v_0 \leq a-1$ and $\beta \leq w_0 \leq 1$. We denote by $(u,v,w)$ the corresponding solution. Then our first result describes a lower spreading estimate for the sum $u+v$.
\begin{proposition}\label{PROP}
Let $(u_0,v_0,w_0) \in X_0$ be such that $u_0+v_0\not\equiv 0$. Then the corresponding solution $(u,v,w)$ satisfies, for each $c\in (0,c_v^*)$,
\begin{equation*}
\liminf_{t\to +\infty}\inf_{|x|\leq ct} \left(u(t,x)+v(t,x)\right)>0.
\end{equation*}
\end{proposition}
Note that the above proposition proves, as a special case, part~(i) of Theorem \ref{THEO-inside1}.
\begin{proof}
We focus on the interval $[0,+\infty)$, as the estimate on $(-\infty,0]$ can be obtained similarly. We shall make use of the persistence lemma with $\zeta=\xi=1$ and $c_2=0$ while $0<c_1 = c_v^*$, see Lemma \ref{LE-weak-strong}.

Due to Lemma \ref{LE-weak-strong}, it is sufficient to show that for any $c\in [0,c_v^*)$ there exists $\varepsilon(c)>0$ such that
for each $(\tilde u_0,\tilde v_0,\tilde w_0)\in \omega_0(0,c_v^*)$ (recall Definition~\ref{definition-omega0}) with $\tilde u_0+\tilde v_0\not\equiv 0$ one has
\begin{equation}\label{az-bis}
\limsup_{t\to +\infty} (\tilde u (t,ct) +\tilde v (t,ct)) \geq \varepsilon(c),
\end{equation}
To prove \eqref{az-bis} we argue by contradiction by fixing $c\in [0,c_v^*)$ and assuming that there exists
a sequence $\tilde U_{0,n}=(\tilde u_{0,n},\tilde v_{0,n},\tilde w_{0,n})\in \omega_0(0,c_v^*)$ with $\tilde u_{0,n}+\tilde v_{0,n}\not\equiv 0$ for all $n\geq 1$ and
such that the corresponding solution $(\tilde u_{n},\tilde v_{n},\tilde w_{n})$ satisfies
\begin{equation*}
\limsup_{t\to  +\infty} \left(\tilde u_n(t,ct)+\tilde v_n(t,ct)\right)\leq \frac{1}{n},\;\forall n\geq 1.
\end{equation*}
This in particular means that for each $n\geq 1$ there exists $t_n$ such that
\begin{equation*}
\left(\tilde u_n+\tilde v_n\right)(t + t_n,c( t + t_n))\leq \frac{2}{n},\;\forall n\geq 1,\;\forall t\geq 0.
\end{equation*}
From the above inequality coupled with the strong maximum principle, we have possibly up to a sub-sequence that
\begin{equation}\label{aze-bis}
\left(\tilde u_n,\tilde v_n,\tilde w_n\right)(t+t_n, x+c(t+t_n))\to (0,0,1)\text{ as $n\to +\infty$},
\end{equation}
uniformly for $t\geq 0$ and locally uniformly for $x\in\R$.

Now, recalling that $0\leq c<c_v^*\leq c_u^*$, let $\delta>0$ be small enough such that
\begin{equation*}
\begin{split}
&r_1(a-1-\delta)-\frac{c^2}{4d_1}>0,\\
&r_2(a-1-\delta)-\frac{c^2}{4d_2}>0,
\end{split}
\end{equation*}
and, due to Lemma \ref{LE-eig}, choose $R>0$ large enough such that the principal eigenvalue problems,
\begin{equation*}
\begin{cases}
-d_i\varphi_i''(x)-c\varphi_n'(x)- r_i(a-1-\delta)\varphi_i(x)=\lambda^i_{R} \varphi_i(x),\;x\in (-R,R),\\
\varphi_i(\pm R)=0\text{ and }\varphi_i>0\text{ on $(-R,R)$},
\end{cases}
\end{equation*}
satisfy $\lambda_R^i <0$ for $i=1,2$.

Set, for $n\geq 1$,
$$
(u_n,v_n,w_n)(t,x):=\left(\tilde u_n,\tilde v_n,\tilde w_n\right)(t+t_n, x+c(t+t_n)).
$$
Since $u_n+v_n>0$, assume for instance that $u_n>0$.
Due to \eqref{aze-bis} there exists $n\geq 1$ large enough such that
\begin{equation}\label{az-w}
a w_n(t,x)-kv_n (t,x) -u_n (t,x)-1\geq a-1-\delta,\;\forall t\geq 0,\; \forall x\in [-R,R].
\end{equation}
Next, using \eqref{az-w} the function $u_n$ satisfies the following inequality for $t\geq 0$ and $x\in [-R,R]$:
\begin{equation*}
(u_n)_t \geq d_1 (u_n)_{xx} +c (u_n)_x +r_1 (a-1-\delta)u_n.
\end{equation*}
On the other hand, for any $\epsilon >0$, the function $\underline{u}(t,x):= \epsilon e^{-\lambda_R^1t}\varphi_1(x)$ becomes a sub-solution of the above equation.
Hence since $u_n>0$ and $\underline{u}(t,\pm R)=0$, there exists $\epsilon>0$ such that
$$
u_n(t,x) \geq \epsilon e^{-\lambda_R^1t}\varphi_1(x),\;\forall t\geq 0,\; \forall x\in [-R,R].
$$
Since $\lambda^1_{R}< 0$ one obtains that $\tilde u_n$ is unbounded for large time, which contradicts Proposition~\ref{prop:prelim_estim}.
The remaining case when $u_n \equiv 0$ and $v_n >0$ can be treated similarly and this completes the proof of \eqref{az-bis}.
As already mentioned above this also completes the proof of Proposition \ref{PROP} using Lemma \ref{LE-weak-strong} with $c_2=0<c_1 =c_v^*$ and $\zeta=\xi=1$.
\end{proof}
From the above proposition we shall now derive some new estimates for each component $u$ and $v$.
\begin{proposition}\label{PROP1}
Let $(u_0,v_0,w_0)\in X_0$ be given and $(u,v,w)$ be the corresponding solution. Recalling the definition of $c_u^{**}$ and $c_v^{**}$ in \eqref{c**}, the following statements hold:
\begin{itemize}
\item[(i)] if $v_0\not\equiv 0$ then for any $c\in (0,c_v^{**})$ one has
\begin{equation*}
\liminf_{t\to +\infty} \inf_{|x|\leq ct} v(t,x)>0;
\end{equation*}
\item[(ii)] if $u_0\not\equiv 0$ then for any $c\in \left(0,\min\left(c_u^{**},c_v^*\right)\right)$ one has
\begin{equation*}
\liminf_{t\to +\infty} \inf_{|x|\leq ct} u(t,x)>0.
\end{equation*}
\end{itemize}
\end{proposition}
Notice that, when $c_v^* = c_u^*$, then $\min (c_u^{**} ,c_v^* ) = c_u^{**}$. Therefore the above proposition proves part~(ii) of Theorem \ref{THEO-inside1} by choosing more specific initial data.
We also point out that $c_v^{**} < c_v^* \leq c_u^*$, which is why a minimum does not appear in part~(i) of Proposition~\ref{PROP1}; actually both statements can be proved in the same fashion.

\begin{proof}
As mentioned above, we shall only prove (i) since the proof of (ii) is similar.
Moreover, we only focus on the interval $[0,+\infty)$, as the interval $(-\infty,0]$ can be dealt with by a symmetrical argument.
The proof of (i) makes use of Lemma \ref{LE-weak-strong} again, coupled with Proposition~\ref{PROP} and Lemma~\ref{LE_entire} dealing with the description of the uniformly positive entire solutions of the sub-system $(u,w)$,
obtained from \eqref{u-eq}-\eqref{w-eq} with~$v=0$.

Due to Lemma \ref{LE-weak-strong}, to prove (i), it is sufficient to show that for any $c\in [0,c_v^{**} )$ there exists $\varepsilon(c)>0$ such that
for each $(\tilde u_0,\tilde v_0,\tilde w_0)\in \omega_0(0,c_v^{**})$ with $\tilde v_0\not\equiv 0$ one has
\begin{equation}\label{az-ter}
\limsup_{t\to +\infty} \tilde v(t,ct)\geq \varepsilon(c),
\end{equation}
To prove \eqref{az-ter} we argue by contradiction by fixing $c\in [0,c_v^{**})$ and assuming that there exist
a sequence $\tilde U_{0,n}=(\tilde u_{0,n},\tilde v_{0,n},\tilde w_{0,n})\in \omega_0(0,c_v^{**})$ with $\tilde v_{0,n}\not\equiv 0$ for all $n\geq 1$,
and a sequence $t_n \to +\infty$ such that the corresponding solution $(\tilde u_{n},\tilde v_{n},\tilde w_{n})$ satisfies
\begin{equation}\label{pa}
\tilde v_n( t + t_n,c(t  + t_n))\leq \frac{2}{n},\;\forall n\geq 1,\;\forall t\geq 0.
\end{equation}
We now consider the sequence of functions $\{(u_n,v_n,w_n)\}$ defined by
$$
(u_n,v_n,w_n)(t,x):=(\tilde u_n,\tilde v_n,\tilde w_n)(t+t_n,x+c(t+t_n)).
$$
Then we claim that:
\begin{claim}\label{claim}
Up to increasing $t_n$, the sequence $\{(u_n,v_n,w_n)\}$ satisfies
\begin{equation*}
\lim_{n\to +\infty}(u_n,v_n,w_n)(t,x)=\left(\tilde p,0,\tilde q\right),
\end{equation*}
uniformly for $t\geq 0$ and locally uniformly for $x\in\R$.
\end{claim}

For Claim~\ref{claim}, by standard parabolic estimates, up to a sub-sequence not relabelled, one may assume that $(u_n,v_n,w_n)$ converges to an entire in time solution $(u_\infty, v_\infty, w_\infty)$.
Recall also that $(u_n,v_n,w_n) (0,\cdot) \in X_0$ thanks to Proposition~\ref{prop:prelim_estim}. Up to increasing $t_n$ and since $c<c_v^{**}$, Proposition \ref{PROP} ensures that there exists $\varepsilon>0$ such that
\begin{equation*}
u_\infty (t,x) + v_\infty (t,x) \geq \varepsilon,\;\forall (t,x) \in\R^2,
\end{equation*}
while \eqref{pa} together with the strong maximum principle imply that $v_\infty(t,x)\equiv 0$.

As a consequence, each limit function $(u_\infty,0,w_\infty)$ of $(u_n,v_n,w_n)$ for the open compact topology satisfies
\begin{equation*}
\varepsilon\leq u_\infty(t,x)\leq a-1,\;\beta\leq w_\infty(t,x)\leq 1,\;\forall (t,x)\in\R^2.
\end{equation*}
Now since $(u_\infty,w_\infty)(t,x-ct)$ is an entire solution of the sub-system obtained from~\eqref{u-eq}-\eqref{w-eq} with~$v=0$, Lemma \ref{LE_entire} ensures that
$$
(u_\infty,w_\infty)(t,x)\equiv (\tilde p,\tilde q).
$$
To sum-up the above analysis, we have obtained
\begin{equation*}
\lim_{n\to +\infty}(u_n,v_n,w_n)(t,x)=\left(\tilde p,0,\tilde q\right)\text{ locally uniformly for $(t,x)\in \R\times\R$}.
\end{equation*}
Now due to the uniform estimate \eqref{pa} for $t\geq 0$, one obtains using similar arguments that the convergence is uniform for $t\geq 0$ and uniform on the compact sets for $x\in\R$.
Then Claim~\ref{claim} is proved.
%

Equipped with the above claim, we now complete the proof of Proposition~\ref{PROP1}. To that aim and recalling that $c<c_v^{**}$ we fix $\delta>0$ small enough such that
\begin{equation*}
r_2(-1-h\tilde p+a\tilde q-\delta)-\frac{c^2}{4d_2}>0,
\end{equation*}
and $R>0$ large enough such that
\begin{equation*}
-\lambda:=r_2(-1-h\tilde p+a\tilde q-\delta)-\frac{c^2}{4d_2}-\frac{d\pi^2}{4R^2}>0.
\end{equation*}
Using Claim~\ref{claim}, let us fix $n$ large enough such that for all $t\geq 0$ and $x\in [-R,R]$
\begin{equation*}
r_2(-1-hu_n(t,x)-v_n(t,x)+a w_n(t,x))\geq r_2(-1-h\tilde p+a\tilde q-\delta).
\end{equation*}
As in the proof of Proposition~\ref{PROP}, we apply the parabolic comparison principle to obtain that there exists $\epsilon>0$ such that
\begin{equation*}
v_n (t,x) \geq \epsilon e^{-\lambda t}\varphi (x),\;\forall t\geq 0,\, \forall x\in [-R,R],
\end{equation*}
wherein $\varphi$ denotes a positive principal eigenfunction of
\begin{equation*}
\begin{cases}
-d_2\varphi''(x) -c\varphi'(x) - r_2(-1-h\tilde p+a\tilde q-\delta)\varphi (x) =\lambda\varphi(x),\;x\in (-R,R),\\
\varphi(\pm R)=0\text{ and }\varphi>0\text{ on }(-R,R).
\end{cases}
\end{equation*}
Here again, since $\lambda < 0$, this lower estimate contradicts the boundedness of the function~$v_n$. This completes the proof of the proposition. \end{proof}

Finally, part~(iii) of Theorem \ref{THEO-inside1} directly follows from the above proposition coupled with Lemma~\ref{LE-entire-full}. This ends the proof of Theorem \ref{THEO-inside1}.
Furthermore, as an additional corollary of Proposition~\ref{PROP1} coupled with Lemmas~\ref{LE_entire} and~\ref{LE-entire-full}, we also obtain the following important result which shall be crucially used in the next sections.
\begin{corollary}\label{COR-conv}
Let $(u_0,v_0,w_0)\in X_0$ be given and $(u,v,w)$ be the corresponding solution of \eqref{u-eq}-\eqref{w-eq}.
\begin{itemize}
\item[(i)] If $u_0\not\equiv 0$ and $v_0 \equiv 0$ then
\begin{equation*}
\lim_{t\to+\infty} (u,v,w)(t,x)=(\tilde p,0,\tilde q)\text{ locally uniformly for $x\in\R$}.
\end{equation*}
\item[(ii)] If $u_0\not\equiv 0$ and $v_0\not\equiv 0$ then
\begin{equation*}
\lim_{t\to+\infty} (u,v,w)(t,x)=(u^*,v^*,w^*)\text{ locally uniformly for $x\in\R$}.
\end{equation*}
\end{itemize}
Herein $(\tilde p,0,\tilde q)$ and $(u^*,v^*,w^*)$ are respectively defined in \eqref{def-pq} and \eqref{def-equi}.
As a special case, as soon as $u_0\not\equiv 0$ then the solution satisfies
\begin{equation*}
\lim_{t\to+\infty} u(t,0)\geq \min\left(\tilde p, u^*\right).
\end{equation*}
\end{corollary}


\section{Two competitors: proof of Theorem \ref{inside}}\label{sec:competitors2}

In this section we complete the proof of Theorem \ref{inside}. Here recall that $\mu=0$ and we consider the case $d:= d_1=d_2$ and $r:= r_1=r_2 $, so that $c^* := c_u^*=c_v^*$.

We first introduce the following function
\beaa
U:=u-\kappa v,\; \kappa=\frac{1-k}{1-h}.
\eeaa
It is easy to check that $U$ satisfies
\be\label{U-eq}
U_t=dU_{xx}+r\{-1+aw-[(2-h)\kappa+k]v-U\}U,\; t>0,\, x\in\R.
\ee
From \eqref{U-eq} and the maximum principle, we have

\begin{lemma}\label{p-est}
Assume that $d_1=d_2$ and $r_1=r_2$.
Suppose that $u_0\le \kappa v_0$ in $\R$.
Then $u(t,x)\le \kappa v(t,x)$ for all $t>0$, $x\in\R$.
\end{lemma}

\begin{proof}
Suppose that $u_0\le\kappa v_0$ in $\R$. Then $U(0,x)\le 0$ for all $x\in\R$. It follows from the maximum principle and \eqref{U-eq} that $U\le 0$ for all $t>0$.
Hence the lemma is proved.
\end{proof}

As a corollary, we obtain the following exact spreading speed of $v$.

\begin{corollary}\label{v-sp1}
Assume that $d_1=d_2$ and $r_1=r_2$. Suppose that $u_0\le\kappa v_0$ in $\R$.
Then $v$ spreads at the speed $c_v^*=c^*_u$.
\end{corollary}

\begin{proof}
For any $c\in[0,c^*_v)$, it follows from Lemma~\ref{p-est} that
\beaa
\liminf_{t \to +\infty} \inf_{|x|\le ct} v(t,x)\ge\frac{1}{1+\kappa}\liminf_{t \to +\infty} \inf_{|x|\le ct}  (u(t,x) + v(t,x) ).
\eeaa
Hence due to part (i) of Theorem \ref{THEO-inside1} there exists $\varepsilon_0>0$ such that
$$
\liminf_{t \to +\infty} \inf_{|x|\le ct} (u(t,x) + v(t,x) )\geq \varepsilon_0,
$$
so that
\beaa
\liminf_{t \to +\infty} \inf_{|x|\le ct} v(t,x)\ge\frac{1}{1+\kappa}\liminf_{t \to +\infty} \inf_{|x|\le ct} (u(t,x) + v(t,x))\ge\frac{\varepsilon_0}{1+\kappa}>0.
\eeaa
The corollary follows by combining this with \eqref{v-beyond}.
\end{proof}

Alternatively, if $u_0\ge\kappa v_0$ in $\R$, then $u$ spreads at the speed $c^*=c_u^*$, when $d_1=d_2$ and $r_1=r_2$.
{This completes the proof of Theorem~\ref{inside}.}




\section{Two competitors: the case when $c_u^* > c_v^*$}\label{competitors_faster}

In this section, we consider the `two competitors' case, i.e., $\mu = 0$ and we aim at proving Theorems~\ref{th:middle_zone} and~\ref{th-NP}.
We assume throughout this section that $c_u^* > c_v^*$, so that an intermediate zone may appear where the predator~$u$ invades ahead of the predator~$v$.

\subsection{Spreading of the faster predator}

This subsection is concerned with the proof of Theorem \ref{th:middle_zone}. The main step here is the following proposition, which together with results from previous sections shall immediately imply Theorem \ref{th:middle_zone}.
\begin{proposition}\label{PROP_inf}
Recall that $c_v^*<c_u^*$. Let $(u_0,v_0,w_0)\in X_0$ be such that $u_0\not \equiv  0$ and $v_0$ are both compactly supported. Then the corresponding solution $(u,v,w)$ satisfies for each $c_v^*<c_2<c_1<c_u^*$
\begin{equation*}
\lim_{t\to+\infty} \sup_{c_2t\leq x\leq c_1t}\left\{ |u(t,x)-\tilde p|+v(t,x)+|w(t,x)-\tilde q|\right\}=0,
\end{equation*}
and
\begin{equation}\label{u_inf}
\liminf_{t\to+\infty} \inf_{0\le x\leq c_1 t} u(t,x)>0.
\end{equation}
\end{proposition}
\begin{proof}
To prove the first part of this proposition, we shall make use of Lemma~\ref{LE-weak-strong} while the proof of the second part shall follow from Corollary \ref{COR-conv} above.

We start with the intermediate zone where the fast predator~$u$ shall eventually persist and the slower one~$v$ goes to extinction, and we fix $c_v^* < c_2 < c_1 < c_u^*$.
We shall show that for each $c\in [c_2,c_1) $ there exists $\varepsilon(c)>0$ such that for any $(\tilde u_0,\tilde v_0,\tilde w_0)\in \omega_0(c_2,c_1)$ with $\tilde u_0 \not\equiv 0$,
the corresponding solution $(\tilde u , \tilde v ,\tilde w)$ satisfies
\begin{equation}\label{PROP_inf_eq1}
\limsup_{ t \to +\infty}   \tilde u(t,  ct) \geq \varepsilon(c).
\end{equation}

{First notice that, from the Definition~\ref{definition-omega0} of $\omega_0 (c_2, c_1)$, the function $\tilde{v}$ is either a finite time shift of $v$ (whose initial condition is compactly supported),
or by Theorem~\ref{theo:spread_plusplus} and the fact that $c_2 > c_v^*$, it satisfies $\tilde{v} \equiv 0$. In both cases, another use of Theorem~\ref{theo:spread_plusplus} ensures that,} for any $c > c_v^*$,
$$\lim_{ t \to +\infty} \sup_{x \geq ct } \tilde v(t,x) = 0.$$
As a consequence one in particular has
\begin{equation}\label{v-spread}
 \lim_{t \to +\infty} \tilde v (t,ct)=0,\;\;\forall c>c_v^*.
\end{equation}
Let us fix $c \in (c_v^*, c_u^*)$. We proceed again by contradiction to prove \eqref{PROP_inf_eq1} and assume that
$$\limsup_{t \to +\infty}\, \tilde u_n (t,ct)  \leq \frac{1}{n},$$
for a sequence of solutions $\{(\tilde u_n, \tilde v_n, \tilde w_n)\}$ associated with initial data $\{(\tilde u_{0,n} , \tilde v_{0,n} , \tilde w_{0,n})\}$ in $\omega_0(c_2,c_1)$,
where $\tilde u_{0,n}\not \equiv  0$ for any $n \geq 1$.
Then, for each $n \geq 1$, there exists $t_n$ large enough such that
$$\tilde u_n (t, ct)   \leq \frac{2}{n},\;\forall\, t \geq t_n.$$

Using~\eqref{v-spread}, we can increase~$t_n$ so that also $t_n \to +\infty$ and $\tilde v_n (t,ct) \leq \frac{2}{n}$ for all $t \geq t_n$.
In particular, passing to the limit as $n \to +\infty$ and applying a strong maximum principle, one may check that for any $R >0$,
\begin{equation}\label{claim_weakspread_v2}
\limsup_{n \to +\infty} \sup_{t \geq t_n , |x -ct | \leq R } ( \tilde u_n (t,x) + \tilde v_n (t,x) ) = 0.
\end{equation}
We then claim that, for any $R >0$,
\begin{equation}\label{claim_weakspread_w2}
\limsup_{n \to +\infty} \sup_{ t \geq t_n , |x - ct| \leq R} | \tilde w_n (t,x)  -1 | = 0.
\end{equation}
The proof is precisely the same as that of \eqref{claim_weakspread_w1}, so we omit it here.

Now, for any small $\delta >0$ and large $R>0$, we can take $n$ large enough so that, thanks to~\eqref{claim_weakspread_v2} and~\eqref{claim_weakspread_w2},
$$(\tilde u_n)_t \geq d_1 (\tilde u_n)_{xx} + r_1 \tilde u_n (a-1 -\delta ) ,$$
for all $t \geq t_n$ and $|x - ct_n | \leq R$. Proceeding as in the proof of Proposition~\ref{PROP1}, we find some $\epsilon >0$ such that
$$\tilde u_n (t,x + ct)  \geq \epsilon e^{-\lambda t} \varphi (x),$$
for all $t \geq t_n$ and $|x| \leq R$, where
$$ -\lambda = r_1 (a-1 - \delta) - \frac{c^2}{4d_1} - \frac{d\pi^2}{4 R^2} > 0,$$
and $\varphi$ is the corresponding positive principal eigenfunction from Lemma~\ref{LE-eig}. This contradicts Proposition~\ref{prop:prelim_estim} and the boundedness of the solution.
Thus we have proved~\eqref{PROP_inf_eq1} where $\varepsilon (c)$ does not depend on the initial data.

Next, applying Lemma~\ref{LE-weak-strong} with $\zeta = 1$ and $\xi =0$, we obtain that for each $c_v^*<c_2< c<c_1<c_u^*$,
\begin{equation*}
\liminf_{t\to+\infty} \inf_{c_2 t \leq x \leq ct} u(t,x)>0.
\end{equation*}
Recalling \eqref{v-beyond} from Theorem~\ref{theo:spread_plusplus} and Lemma \ref{LE_entire} this completes the first part of Proposition~\ref{PROP_inf}.

We now turn to the second part of the proposition, which shares some similarity with the second step of the proof of Lemma~\ref{LE-weak-strong}.
To do so let $c_1\in (c_v^*,c_u^*)$ be given, and assume by contradiction that there exist sequences $t_n \to +\infty$ and $c_n \in [0,c_1)$ such that
\begin{equation*}
\lim_{n\to +\infty} u(t_n, c_n t_n)=0.
\end{equation*}
Without loss of generality, up to a sub-sequence, we assume that $c_n \to c\in [0,c_1]$. Choose $c'$ such that $c_1 <c'<c_u^*$
and define the sequence
$$
t'_n := \frac{c_n t_n}{c'}\in [0,t_n),\;\forall n\geq 0.
$$

Consider first the case when the sequence $\{c_n t_n\}$ is bounded, which may happen only if $c=0$. Then up to extraction of a sub-sequence, one has as $n \to +\infty$ that
$$c_n t_n \to x_\infty \in \R,$$
and, due to the strong maximum principle,
$$
\lim_{n\to +\infty}u(t + t_n , x + c_n t_n )=0\text{ locally uniformly for $(t,x)\in\R^2$}.
$$
This implies in particular that $u(t_n,0) \to 0$ as $n \to +\infty$, which contradicts Corollary~\ref{COR-conv}.

Next, we consider the case when $\{c_nt_n\}$ has no bounded sub-sequence. In particular, we assume below that $t_n'\to + \infty$ as $n\to + \infty$.
Set $\varrho:=\min (u^*,\tilde p)$. Then due to the first part of Proposition~\ref{PROP_inf}, since $c'\in (c_v^*,c_u^*)$ we have for all large $n$ that
$$u(t_n ', c_n t_n) = u (t_n ', c' t_n ')> \frac{3}{4}\varrho.$$
Then we introduce a third time sequence $\{ t''_n \}$ with
$$ t''_n := \inf \left\{ t \leq  t_n \, | \ \forall s \in (t , t_n), \quad u (s, c_n t_n )  \leq \frac{ \varrho}{2} \right\} \in (t' _n, t_n).$$
Since $u(t_n , c_n t_n) \to 0$ as $n \to +\infty$, we get
$$u (t_n '' , c_n t_n) = \frac{\varrho}{2},$$
and, as before, by a limiting argument and a strong maximum principle, that
$$t_n - t_n '' \to +\infty,$$
as $n \to +\infty$.
Again, by parabolic estimates and up to extraction of a sub-sequence, we find that $(u,v,w)(t + t_n '' ,x + c_n t_n )$ converges to a solution $(u_\infty, v_\infty , w_\infty)$ of \eqref{u-eq}-\eqref{w-eq} that satisfies
\begin{equation*}
u_\infty(0,0)= \frac{\varrho}{2}\text{ and }
u_\infty(t,0)\leq  \frac{\varrho}{2},\;\forall t\geq 0.
\end{equation*}
Recalling the definition of $\varrho$ above, and noticing that $(u_\infty, v_\infty, w_\infty) (0,\cdot) \in X_0$, this again contradicts Corollary \ref{COR-conv} and completes the proof of Proposition~\ref{PROP_inf}.
\end{proof}
Parts (i) and (ii) of Theorem~\ref{th:middle_zone} follow immediately from Proposition~\ref{PROP_inf}, using here again a symmetry argument to handle negative $x$.
Finally, applying Propositions~\ref{PROP1} and~\ref{PROP_inf} coupled with Lemma~\ref{LE-entire-full}, we also obtain part~(iii) of Theorem~\ref{th:middle_zone}.


\subsection{Counter-example: nonlocal pulling}

The remaining question is the spreading speed of $v$.
From Theorem~\ref{th:middle_zone} and \eqref{c**}, we might wonder whether it is $c_v^{**}$. However, it is known that a `nonlocal pulling' phenomenon may occur in competition systems.
`Nonlocal pulling' here refers to the fact that the zone ahead of the point $c_u^* t$, where~$u$ is close to 0, may have an effect on the speed of its competitor~$v$.
This may be surprising because $c_u^*$ is strictly larger than $c_v^*$, thus strictly larger than the spreading speed of~$v$. We refer to~\cite{GL} for an example of such a situation in the two species competition system.

Now, we give a short proof that the spreading speed of $v$ may indeed be strictly larger than $c_v^{**}$ in our context of a predator-prey system with two predators.
We start by constructing a subsolution for the $v$-equation, moving at a speed larger than $c_v^{**}$. In what follows, we let $\varepsilon >0$ be arbitrarily small.

Due to Theorem~\ref{theo:spread_plus} and Proposition~\ref{prop:prelim_estim}, we can assume up to some shift in time that
$$w (t,x) \geq \underline{w} (t,x),\quad u (t,x) \leq \overline{u} (t,x),$$
where
\beaa
&&\underline{w} (t,x) := \left\{
\begin{array}{ll}
\beta & \mbox{ if } x < (c_u^* + \varepsilon) t,\\
\displaystyle  1-  \frac{\varepsilon}{a}  & \mbox{ if } x \geq (c_u^* + \varepsilon) t,
\end{array}
\right.\\
&&\overline{u} (t,x) := \left\{
\begin{array}{ll}
a-1 & \mbox{ if } x < (c_u^* + \varepsilon ) t,\\
\displaystyle \frac{\varepsilon}{h} & \mbox{ if } x \geq (c_u^*+ \varepsilon) t.
\end{array}
\right.
\eeaa
Then $v$ satisfies
$$v_t \geq d_2 v_{xx} + r_2 v (-1 - h \overline{u} -v + a\underline{w}).$$
Up to linearization around 0, this leads us to look for a subsolution of
$$v_t = d_2 v_{xx} + r_2 v (-1 - h \overline{u} - \varepsilon + a \underline{w}).$$
If we find such a (bounded) subsolution, then it is a subsolution of the previous nonlinear equation up to multiplying by a small enough constant.
Because of the discontinuity at $x =( c_u^* + \varepsilon) t$, we shall `glue' two ansatzes. For ease of notation, let us denote $c_\varepsilon := c_u^* + \varepsilon$.

Let us first deal with the ansatz in the moving frame with speed $c_\varepsilon$. We look for a solution of the above linear equation of the type
$$\underline{v_1} (t, x + c_\varepsilon t) := e^{-r t} e^{-\nu x} \sin (\omega x),$$
whose support is $(0, \pi / \omega)$, for some constants $r$, $\nu$ and $\omega$. In particular, it is included in the part where $\underline{w} = 1- {\varepsilon}/{a}$ and $\overline{u} = {\varepsilon}/{h}$.
Therefore, putting this into the linearized equation, we find that $r$, $\nu$ and $\omega$ satisfy the system
\beaa
\begin{cases}
c_\varepsilon - 2 d_2 \nu = 0,\\
- d_2 \omega^2 + d_2 \nu^2 - c_\varepsilon \nu + r_2 (a-1 - 3 \varepsilon) = -r.
\end{cases}
\eeaa
Thus
$$\nu = \frac{c_u^* + \varepsilon}{2 d_2}$$
and, taking $\omega$ very small, say $\omega = \varepsilon$, we find a compactly supported subsolution in the moving frame with speed $c_\varepsilon$ which converges to 0 exponentially in time at rate
\begin{equation}\label{nlp_plusplus}
{r} = \frac{c_\varepsilon^2}{4d_2} + d_2 \varepsilon^2 - r_2 (a-1 - 3 \varepsilon) >0.
\end{equation}
Here the positivity of $r$ comes from the fact that $c_\varepsilon \geq c_u^* > 2 \sqrt{d_2 r_2 (a-1)} = c_v^*$. More precisely, we have
$$\underline{v_1} (t, 1 + c_\varepsilon t) = C_\varepsilon  e^{-r t},$$
for some $C_\varepsilon >0$.

Now let us turn to the second ansatz. Here we look for a subsolution of the type
$$ \underline{v_2} := \max\{ 0, A  e^{-\lambda (x - c t)} - B e^{- (\lambda + \eta) (x-ct)} \},$$
where constants $A>0$, $B>0$, $\lambda>0$, $\eta>0$ and $c \in (c_v^{**} , c_u^*)$ are to be determined. It is enough to show that $\underline{v_2}$ satisfies
$$v_t \leq d_2 v_{xx} + r_2 v (a\beta -1 -h (a-1) - \varepsilon ).$$
This is a rather standard construction. We choose $\lambda$ as the smaller positive solution of
$$d_2 \lambda^2 - c \lambda + r_2 (a \beta - 1 - h (a-1) - \varepsilon) = 0, $$
i.e.,
$$\lambda= \lambda (c,\varepsilon) := \frac{c - \sqrt{c^2 - 4 d_2 r_2 (a\beta - 1 - h (a-1) - \varepsilon)}}{2 d_2}.$$
Notice that this is possible provided that
$$c >2 \sqrt{d_2 r_2 (a \beta - 1 -h (a-1) - \varepsilon)} >0 .$$
Thus we assume that this middle term is positive. Recalling that $\beta = 1 - 2 (a-1) b$, $a>1$ and as $\varepsilon$ can be arbitrarily small, this rewrites as
\begin{equation}\label{cond_nlp}
{1 - 2 ab - h >0.}
\end{equation}
{This is true if, for instance, $h$ and $b$ are small.} In particular, it is compatible with our other assumptions in \eqref{c1}. Taking $\eta$ small enough, we get that
$$d_2 (\lambda + \eta)^2 - c (\lambda +\eta)+ r_2 (a \beta - 1 - h (a-1) - \varepsilon) < 0. $$
Then $\underline{v_2}$ can be checked to be a subsolution, with an appropriate choice of $A,B>0$.

Finally, we want to glue the two ansatzes to find a compactly supported subsolution. On the one hand, $\underline{v_2}$ has finite support to the left.
On the other hand, $\underline{v_1}$ has compact support by definition. Therefore, it is possible to glue them into a subsolution, at least for $t$ large enough, provided that
\begin{equation}\label{nlc_plus}
\lambda (c,\varepsilon) (c_\varepsilon - c)  > r.
\end{equation}
Indeed, this guarantees that $\underline{v_2}$ converges to 0 as $t \to +\infty$ and in the moving frame with speed $c_\varepsilon$, faster than $\underline{v_1}$.
Thus they intersect twice and a subsolution can be constructed.

Now notice that $$c  \mapsto \lambda (c,\varepsilon) ( c_\varepsilon  -c)$$ is a decreasing function in the interval $ [2 \sqrt{d_2 r_2 (a \beta - 1 - h(a-1 ) - \varepsilon)} , c_\varepsilon]$.
Therefore there exists a critical $c$ below which inequality \eqref{nlc_plus} holds, and above which it does not.
It is straightforward to check that $c_v^{**}$ belongs to this interval, and the question is now whether $c_v^{**}$ is below or above this critical $c$. If it is below, in other words if
$$\lambda (c_v^{**},\varepsilon) (c_\varepsilon - c_v^{**}) > r,$$
then nonlocal pulling occurs and $v$ spreads at a speed strictly faster than $c_v^{**}$. Since $\varepsilon$ is arbitrarily small, it is enough that
$$\lambda (c_v^{**},0) (c_u^* - c_v^{**}) > r .$$
Recalling~\eqref{nlp_plusplus} with $\varepsilon=0$, this rewrites as
\begin{equation}\label{nlp_last}
\left[c_v^{**} - \sqrt{\left(c_v^{**}\right)^2 - 4 d_2 r_2 (a\beta - 1 - h (a-1))}\right] (c_u^* - c_v^{**} ) >  {\frac{\left(c_u^*\right)^2 - \left(c_v^*\right)^2}{2} }.
\end{equation}
Together with \eqref{cond_nlp} this corresponds to the parameter conditions arising in Theorem~\ref{th-NP}.
Note that the subsolution constructed here is sufficient to conclude that
$$
\liminf_{t\to +\infty} v(t,ct)>0.
$$
Moreover the uniform positivity of $v$ for $|x|\leq ct$ as $t \to +\infty$ follows from Corollary~\ref{COR-conv} (for $v$), the fact that $u$ spreads with speed $c_u^*$,
and the same argument as the one developed to prove \eqref{u_inf}.

Now to conclude, let us just try to find one situation which is compatible with previous assumptions.
To do this, let us simply vary the parameter $d_1$ (so that properties of the ODE system and condition~\eqref{cond_nlp} remain unchanged) and consider the case when $d_1\to  \left(\frac{d_2 r_2}{r_1}\right)^+$
so that $0<c_u^*-c_v^*\to 0$.
In that case, the right hand term of \eqref{nlp_last} converges to $0$ while the left hand term converges to
\begin{equation*}
\left[c_v^{**} - \sqrt{\left(c_v^{**}\right)^2 - 4 d_2 r_2 (a\beta - 1 - h (a-1))}\right]( c_v^* -c_v^{**} ) >0 .
\end{equation*}
Therefore, under condition \eqref{cond_nlp} and if $d_1 > {d_2 r_2}/{r_1}$ but close enough, we find that $v$ spreads strictly faster than $c_v^{**}$.

\noindent{\bf Acknowledgements}

\noindent{Part of this work was carried out during the first two authors were visiting Tamkang University in 2019.
We would like to thank the support of the National Center for Theoretical Sciences (NCTS) for their visits.
This work was carried out in the framework of the International Research Network ``ReaDiNet'' jointly funded by CNRS and NCTS.
The third author is partially supported by the Ministry of Science and Technology of Taiwan under the grant 108-2115-M-032-006-MY3.
The fourth author is supported by JSPS KAKENHI Grant-in-Aid for Young Scientists (B) (No.16K17634) and JSPS KAKENHI Grant-in-Aid for Scientific Research (C) (No.20K03708).}



\end{document}